\documentclass[10pt]{amsart}

\usepackage{amsmath, amsthm, amssymb, xcolor, graphicx, enumitem, 
everypage, datetime, array} 
\usepackage{wrapfig}

\newcounter{citedtheorems}

\newtheorem{defn}{Definition}[section]
\newtheorem{theorem}[defn]{Theorem}

\newtheorem*{theorem-m}{Theorem \ref{main-theorem}}
\newtheorem*{thm-p2a}{Theorem \ref{t:p2a}}

\newtheorem*{thm-seq}{Theorem \ref{t:seq}}
\newtheorem*{thm-e}{Theorem}
\newtheorem*{thm-m}{Main Theorem}
\newtheorem*{theorem-abs1}{Theorem \ref{ind-theorem}}
\newtheorem*{theorem-abs2}{Theorem \ref{a23}}
\newtheorem*{theorem-abs3}{Theorem \ref{ind-new}}
\newtheorem*{theorem-abs4}{Theorem \ref{m1}}
\newtheorem*{thm-x}{Theorem}
\newtheorem{thm-lit}[citedtheorems]{Theorem}
\newtheorem{defn-lit}[citedtheorems]{Definition}
\newtheorem{fact-lit}[citedtheorems]{Fact}
\newtheorem{fact}[defn]{Fact}

\newtheorem{defn-claim}[defn]{Definition/Claim}

\newtheorem*{defn-in}{Definition \arabic{section}.\arabic{equation}}

\newtheorem*{claim-in}{Claim \arabic{section}.\arabic{equation}}

\newtheorem{concl}[defn]{Conclusion}
\newtheorem{conv}[defn]{Convention}
\newtheorem{claim}[defn]{Claim}

\newtheorem{lemma}[defn]{Lemma}
\newtheorem{obs}[defn]{Observation}
\newtheorem{rmk}[defn]{Remark}

\newtheorem{ntn}[defn]{Notation}
\newtheorem{disc}[defn]{Discussion}

\newtheorem{qst}[defn]{Question}

\newcommand{\lost}{\L os' }
\newcommand{\los}{\L os }
\newcommand{\br}{\vspace{2mm}}
\newcommand{\bbr}{\vspace{3mm}}

\newcommand{\sbr}{\vspace{1mm}}
\newcommand{\mfa}{\mathfrak{a}}
\newcommand{\mfb}{\mathfrak{b}}

\newcommand{\bM}{\mathbf{M}}

\newcommand{\vt}{\vartheta}

\newcommand{\kim}{(\kappa, \mci, \bar{m})}
\newcommand{\tfeq}{T_{\operatorname{feq}}}

\newcommand{\mk}{\mathbf{k}}

\newcommand{\mcq}{\cQ}
\newcommand{\mcp}{\cP}

\newcommand{\lex}{\operatorname{lex}}

\newcommand{\bp}{\leq_{\operatorname{proj}}}
\newcommand{\proj}{\operatorname{proj}}

\newcommand{\gee}{\mathcal{G}}

\newcommand{\ml}{\mathcal{L}}
\newcommand{\tlf}{\trianglelefteq}

\newcommand{\rn}{\operatorname{range}}

\newcommand{\dom}{\operatorname{dom}}

\newcommand{\uu}{\mathcal{U}}

\newcommand{{\xw}}{\mathbf{w}}

\newcommand{\vv}{\mathcal{V}}

\newcommand{\cP}{\mathcal{P}}
\newcommand{\cQ}{\mathcal{Q}}

\newcommand{\del}{\partial}

\newcolumntype{L}{>{\centering\arraybackslash}m{4cm}}

\newcommand{\AP}{\operatorname{AP}}
\newcommand{\xn}{\mathfrak{n}}
\newcommand{\bk}{\mathbf{k}}

\newcommand{\ii}{\mathbf{i}}

\newcommand{\cl}{\operatorname{cl}}

\newcommand{\mct}{\mathcal{T}}

\newcommand{\lgn}{\operatorname{lg}}
\newcommand{\bn}{\mathbf{n}}

\newcommand{\mcm}{\mathcal{M}}
\newcommand{\mcn}{\mathcal{N}}
\newcommand{\mcy}{\mathcal{Y}}

\newcommand{\de}{\mathcal{D}}

\newcommand{\trv}{\mathbf{t}}

\newcommand{\jj}{\mathbf{j}}

\newcommand{\ba}{\mathfrak{B}}

\newcommand{\ee}{\mathcal{E}}
\newcommand{\mch}{\mathcal{H}}

\newcommand{\rstr}{\upharpoonright}
\newcommand{\mci}{\mathcal{I}}

\newcommand{\mcs}{\mathcal{S}}

\newcommand{\vp}{\varphi}

\newcommand{\ma}{\mathbf{a}}
\newcommand{\mb}{\mathbf{b}}
\newcommand{\mc}{\mathbf{c}}

\newcommand{\mx}{\mathbf{x}}
\newcommand{\my}{\mathbf{y}}
\newcommand{\mz}{\mathbf{z}}

\newcommand{\fin}{\operatorname{FI}}

\newcommand{\trg}{T_{\mathbf{rg}}}

\newcommand{\xm}{\mathfrak{m}}

\title[Some simple theories from a Boolean algebra point of view.]{Some simple theories from \\ a boolean algebra point of view}

\author{M. Malliaris and S. Shelah}

\thanks{ \emph{Thanks}. 
Research partially supported by NSF CAREER 1553653, NSF 2051825, BSF 3013005232, and ISF grant 1838/19. 
This is paper 1218 in Shelah's list. 
}

\address{Department of Mathematics, University of Chicago, 5734 S. University, Chicago, IL 60637, USA} 
\email{mem@math.uchicago.edu}

\address{Einstein Institute of Mathematics, Edmond J. Safra Campus, Givat Ram, The Hebrew
University of Jerusalem, Jerusalem, 91904, Israel, and Department of Mathematics,
Hill Center - Busch Campus, Rutgers, The State University of New Jersey, 110
Frelinghuysen Road, Piscataway, NJ 08854-8019 USA}
\email{shelah@math.huji.ac.il}
\urladdr{http://shelah.logic.at}

\begin{document}

\begin{abstract}  
We find a strong separation between two natural families of simple rank one theories in Keisler's order: the theories $T_\xm$ reflecting 
graph sequences, which witness that Keisler's order has the maximum number of classes, and the theories 
$T_{n,k}$, which are the higher-order analogues of the triangle-free random graph.  
The proof involves building Boolean algebras and ultrafilters ``by hand'' to satisfy 
certain model theoretically meaningful chain conditions. 
This may be seen as advancing a line of work going back through Kunen's construction of good ultrafilters in ZFC using families of 
independent functions.  We conclude with a theorem on flexible ultrafilters, and open questions.  
\end{abstract}

\maketitle

\vspace{5mm}

This paper is dedicated to Ken Kunen.   His contributions to set theory and general topology are many, momentous and deep. 

Our general topic is the structure of Keisler's order on certain families of simple rank one theories.  Keisler's order 
(Definition \ref{d:keisler} below) is a large-scale  
classification program in model theory which builds a framework for comparing the complexity of complete, countable 
theories via saturation of regular ultrapowers.  The natural relation of this work to Kunen's is discussed in \S \ref{s:kunen}.

To motivate our main results here: recently in \cite{MiSh:1167}, we finally proved that Keisler's order has the maximum number of classes, continuum many, already among the 
simple unstable theories.  That paper represented a significant shift in our understanding. It involved not only finding new theories, called $T_\xm$ (the $\xm$ is a parameter, see \S \ref{s:theories} below), but also finding a new method, of building ultrafilters and Boolean algebras together.  This raised two very natural questions.

\begin{qst} \label{q0.1}
Are the theories $T_\xm$ below all earlier nonminimal unstable theories? 
\end{qst}

Recall that there is a minimum unstable class in Keisler's order, which is the class of the random graph (see \cite{mm4}, 5.3). 

\begin{qst} \label{q0.2}
Is the new method of building ultrafilters and Boolean algebras together good just for the specific theories $T_\xm$, or is it a general method? 
\end{qst}

The present paper addresses both questions, though perhaps in a way that opens rather than closes the matter.  We prove that the theories $T_\xm$ are incomparable with the theories $T_{n,k}$, the higher-order versions of the triangle-free random graph, studied by Hrushovski in \cite{h:letter}.  These had essentially been the theories known to be ``near'' the random graph, following \cite{MiSh:1050}. So the new results, Theorems \ref{t:first} and \ref{t:dir2}, show the picture is more multifaceted.  
We do this by means of the new method, which is very encouraging for the second question.  
Finally, since these proofs suggest that indeed, non-free Boolean algebras will be central going forwards, we show that a certain very useful theorem which we had proved earlier for free Boolean algebras works in general, in \S \ref{s:flex-cc}; it is especially nice to include here because of the connection to Kunen. 

The mathematical content of the paper is as follows. 
In \S \ref{s:theories}, we present the two kinds of theories studied in this paper: the theories $T_\xm$ which can ``encode finite combinatorics'' and the theories $T_{n,k}$.  
(These are both families of theories; in this description, we sometimes refer to them in the singular when stating results which apply to all.) In \S \ref{s:prelim} we collect the definitions on ultrafilters, Boolean algebras and separation of variables needed for the proofs.  
In \S \ref{s:first-s} we fix $\bn > \bk \geq 2$ and prove existence of a regular ultrafilter which is good for any $T_\xm$ and
 satisfies a model theoretically interesting chain condition depending on $\bn, \bk$. In \S \ref{s:n-s} we prove that this chain condition is sufficient 
 to block saturation for $T_{\bn, \bk}$.   Together these yield Theorem \ref{t:first}, which says $T_{n,k}$ is not below $T_\xm$ in Keisler's order. 
 In \S \ref{s:second} we prove that it is possible to saturate $T_{n,k}$ while not saturating $T_\xm$. 
This gives Theorem \ref{t:dir2} which says that the $T_{\xm}$ are not below $T_{n,k}$ in Keisler's order. (Thus, combining the two stated 
theorems we have incomparability.)  
In \S \ref{s:flex-cc} we show that no regular ultrafilter built by separation of variables where the Boolean algebra $\ba$
has the $\mu$-c.c. for some regular uncountable $\mu < \lambda$ can be flexible; this gives the definitive statement that 
using a small c.c. can block flexibility, a result proved in various forms (e.g. for free Boolean algebras) in earlier theorems.  
The paper concludes with some open problems.  

We thank the anonymous referee for many helpful comments. 

\section{Some relations to Kunen's work} \label{s:kunen} 

In the sixties it had been famously asked whether there was an 
``outside'' or ``algebraic'' characterization of elementarily equivalence. Keisler answered this question, assuming a case of GCH, by proving that 
$M \equiv N$ if and only if $M, N$ have isomorphic ultrapowers. (This is usually called the Keisler-Shelah isomorphism theorem; 
GCH was later removed by Shelah, but we follow the narrative of 
Keisler's proof.)  Keisler first proved that on any infinite $\lambda$ it was possible to build a so-called \emph{good} regular 
ultrafilter on $\lambda$, assuming $2^\lambda = \lambda^+$. 
He then proved that if $\de$ is a good regular ultrafilter on $\lambda$, and $\lambda \geq |\ml|$, 
the ultrapower $M^\lambda/\de$ is $\lambda^+$-saturated. The equivalence then goes as follows.   Given $M \equiv N$, without loss of generality infinite, choose $\lambda \geq |M| + |N| + |\ml|$ and let $\de$ be a good regular ultrafilter on $\lambda$. Since $\de$ is regular, we know that the 
ultrapowers have the full size of the Cartesian power, so $|M^\lambda/\de| = |N^\lambda/\de| = 2^\lambda$. Since $\de$ is good, 
both ultrapowers are $\lambda^+$-saturated. Again using the assumption of GCH, $2^\lambda = \lambda^+$ so both ultrapowers are elementarily equivalent saturated models of the same size and are therefore isomorphic. 

Kunen then gave a new proof, in ZFC, of the existence of good regular ultrafilters on any infinite cardinal $\lambda$.  
(Keisler had used one instance of GCH in constructing a good ultrafilter on $\lambda$, assuming $2^\lambda = \lambda^+$.)  
We may briefly motivate it as follows.  Suppose we decide to build a regular ultrafilter on $\lambda$ by 
induction. We might start by adding a regularizing family to the initial filter $F_{0}$, as regularity is an existential condition. We might then enumerate $\mcp(\lambda)$ as $\langle X_\beta : \beta < 2^\lambda \rangle$ and try to  build by induction on $\alpha$ an increasing continuous chain of filters so that 
if $\alpha = \beta +1$, $F_\alpha$ contains either $X_\beta$ or its complement, eventually arriving to an ultrafilter $F = F_{2^\lambda}$. 
But we need to make the ultrafilter good, that is, for every decreasing sequence $\bar{A} = \langle A_u : u \in [\lambda]^{<\aleph_0} \rangle$ of members of 
$F$ we have to add a multiplicative refinement (a notion introduced by Keisler for capturing when ultraproducts with be $\lambda^+$-saturated, see \S \ref{s:goodness} below). 
So we can list all the potential possible such $\bar{A}$'s as $\langle \bar{A}^\alpha : \alpha < 2^\lambda \rangle$ and we would like to handle $A^\alpha$ at 
stage $\alpha$. As long as $\alpha < \lambda^+$, Keisler's proof works. But e.g. for $\alpha \geq \lambda^+$ maybe $F_\alpha$ generates an ultrafilter so 
we have no freedom left -- an extreme case, however, a possible one.  Kunen's idea was to restrict the construction so that at all stages $\alpha < 2^\lambda$, many decisions have not been made and we have enough freedom. The right notion was expressed in the language of Engelking-Karlowicz \cite{ek} as follows 
(see also earlier work of Hewitt, Marczewski, Pondiczery \cite{hewitt}, \cite{marczewski}, \cite{pondiczery}, 
Fichtenholz, Kantorovich, Hausdorff \cite{f-k}, \cite{hausdorff} as well as work of Cater-Erd\H{o}s-Galvin \cite{CEG} and Comfort-Negrepontis \cite{cn1}, \cite{cn2}). 
Identifying each set with its characteristic function it becomes a 
function from $\lambda$ to $\{ 0, 1 \}$.  Call a family of functions $\gee = \{ g_\beta : \beta < \kappa \} 
\subseteq {^\lambda \lambda}$ \emph{independent} if for any 
$n < \omega$ and any $\beta_0, \dots, \beta_{n-1} < \kappa $ and any $t_0, \dots, t_{n-1}$ from $\rn(\gee)$ we have that 
$\{ i \in \lambda : \beta_0(i) = t_0 \land \cdots \land \beta_{n-1}(i) = t_{n-1} \} \neq \emptyset$.  
Then, for a filter $D$ on $\lambda$, call a family of functions 
``independent mod D'' if ``$\neq \emptyset$'' can be replaced by ``$\neq \emptyset \mod D$''.  Now we are equipped to build an increasing 
continuous chain of filters $\de_\alpha$ and a decreasing continuous chain of independent families $G_\alpha$ (independent mod $\de_\alpha$) 
at each step making decisions about some finitely many functions from our independent family while maintaining 
the hypothesis that the remainder of the functions stay independent mod the new filter.

In later work, additional points become important,  
such as $\de_\alpha$ being maximal modulo which the remaining family is independent, or the significance of the cofinality of the construction; 
for a full treatment, incorporating subsequent advances over the following decade, see \cite{Sh:c} chapter VI. 

Returning to the isomorphism theorem, one can also see there a strong motivation for Keisler's order. 
Suppose we simply consider the problem of producing regular 
ultrapowers $M^\lambda/\de$ which are $\lambda^+$-saturated, ignoring the final invocation of GCH. 
For some theories, such as algebraically closed fields (or any other uncountably 
categorical theory), the assumption that $\de$ is good is not really needed to produce saturation (by which we shall mean $\lambda^+$-saturation); 
any regular ultrafilter will do. For other theories 
(such as number theory), as observed by Keisler in \cite{keisler}, ``good'' is necessary because for these theories 
any failure of goodness can be coded as a failure of saturation. This leads to the natural question of comparing theories according to whether 
any regular ultrafilter $\de$ which produces saturated ultrapowers of models of $T_2$ also produces saturated ultrapowers of models of 
$T_1$: if so, write $T_1 \tlf T_2$. (One additional ingredient, also due to Keisler \cite{keisler}, 
is that the $\lambda^+$-saturation of the ultrapower is an invariant of the theory, so independent of the choice of model $M$, 
when e.g. $T$ is countable and $\de$ is regular on $\lambda$.) 

\br
The theorems in the present paper grew out of a series of works which have been gradually shifting our understanding 
of the mechanisms of the order (for more details, see \S \ref{s:prelim}). 
One early move, called ``separation of variables'' in \cite{MiSh:999}, was the idea that one can build a regular ultrafilter 
to handle saturation in two stages. 
Essentially, one first chooses any Boolean algebra $\ba$ of size $\leq 2^\lambda$ with maximal antichains of size $\leq \lambda$. 
Then one can build a regular good filter $\de_0$ so that $\mcp(\lambda)/\de_0 \cong \ba$. The core of the problem then shifts to 
a problem of building an 
ultrafilter $\de$ on $\ba$, handling the various images of saturation problems as they appear in $\ba$. Finally, $\de$ 
then combines with $\de_0$ in the natural way to give the final regular ultrafilter on $\lambda$.  
The content of this theorem is that first, we can arrange for the quotient to be isomorphic to a Boolean algebra essentially of our choosing, liberating us 
from simply working on the set of subsets of $\lambda$; and second, that we can do the core work on this quotient 
without losing the needed level of resolution.   

Very recently, this has allowed us to start to see a closer relationship between the 
``set theoretic'' regular ultrafilters and the ``model theoretic'' ultrafilters (i.e. the types we try to realize or omit). In the present work this connection is 
explored and developed through the intermediary of chain conditions having both model theoretic and set theoretic content. The results 
suggest there is much more to say.

\setcounter{tocdepth}{1}
\tableofcontents

\section{Two families of theories} \label{s:theories}

We shall consider two families of theories: the theories $T_\xm$ from our paper \cite{MiSh:1167}, and the theories $T_{n,k}$ studied by 
Hrushovski \cite{h:letter} (see the Appendix there) which later played a key role in \cite{MiSh:1050} for $n=k+1$.  We briefly review both here, starting with the second.

\subsection*{Random hypergraphs with a forbidden configuration.} 
Recall that the (theory of the) model-theoretic random graph is the model completion of the theory of graphs with a symmetric irrreflexive edge relation. 
That is, the language has a single binary relation $R$,\footnote{All our languages 
are assumed to have equality.} and there are axioms saying $R$ is symmetric and irreflexive, that there are infinitely many elements, and  
for each finite $m$ there is an axiom saying that given any two disjoint sets each with $m$ elements, 
there is an $x$ connected to every element in the first set and to no element in the second set. 

So as to minimize subtraction, it will be convenient to write $T_{n,k}$ for the model completion of 
the theory of $(k+1)$-uniform hypergraphs 
in which there do not exist 
$(n+1$) distinct elements of which every $(k+1)$ form an $R$-hyperedge. 
That is:

\begin{defn} For each $n>k\geq 2$, let $T_{n,k}$ be the theory in the language with a $(k+1)$-place relation $R$ 
and axioms saying 
$R$ is symmetric and irreflexive, there are infinitely many elements, and 
 for each finite $m$ there is an axiom saying that given any two disjoint sets each with $m$-many $k$-tuples, 
 there is an $x$ connected to every element in the first set and to no element in the second set 
\emph{if and only if} such an $x$ would not cause the forbidden configuration.\footnote{Call a set $A$ of $n$ elements 
\emph{pre-forbidden} if $R$ holds on every $(k+1)$-tuple of distinct elements of $A$. The axiom for $m = \binom{n}{k}$ amounts to saying that 
$x$ exists \emph{unless} there are $n$ elements forming a pre-forbidden set such that every $k$ of these elements 
occur, under some permutation, as one of the $k$-tuples in the first set. The axioms for $m>\binom{n}{k}$ can defer to the earlier one.}
\end{defn}
Model theorists call $T_{n,k}$ a ``generic'' or ``random'' [we point this common usage out because it may be strange to combinatorialists] ``$(n+1)$-free $(k+1)$-hypergraph,'' e.g. 
$T_{3,2}$ is a random tetrahedron-free three-hypergraph.  
 Remarkably, for $n > k \geq 2$, $T_{n,k}$ is simple rank one, with no forking other than equality; this was proved by Hrushovski \cite{h:letter}. 
 A recent exposition is \cite[Fact 5.14]{MiSh:1149}.

\subsection*{Random graphs over sparse sequences.}  The theories $T_\xm$ are based on random graphs, not hypergraphs.\footnote{An  
extension of the theories to hypergraph sequences was very recently carried out in \cite{MiSh:1206}; however, an analysis of these theories in Keisler's order is not clear.} 
Following \cite[\S 3]{MiSh:1167}, let us 
motivate them as follows by defining some approximations $T^0_{\xm, n}$, $n = 0, 1$. 

\br
\begin{itemize}
\item $T_{\xm, 0}$: Suppose the language has unary predicates $\{ \mcq, \mcp \}$ and a binary relation $R$. Consider universal 
axioms saying that $\mcq, \mcp$ partition the domain and that $R \subseteq \mcq \times \mcp$. 
\\ In (a model $M$ of) the model completion, 
$\mcq^M$ and $\mcp^M$ are infinite, and $R$ is a [model-theoretic] bipartite random graph.

\br
\item $T_{\xm, 1}$: Suppose the language has unary predicates 
\[ \{ \mcq, \mcp, Q_{\langle \rangle}, Q_{0}, Q_{1}, P_{\langle \rangle}, P_{0}, P_{1} \}. \]  
Consider universal axioms saying that $\mcq, \mcp$ partition the domain, that $Q_{\langle \rangle} = \mcq$, $P_{\langle \rangle} = \mcp$; 
that $Q_0, Q_1$ partition $\mcq$, and $P_0, P_1$ partition $\mcp$; that $R \subseteq \mcq \times \mcp$;  and that $R$ is forbidden between 
$Q_0$ and $P_1$. 
\\ In (a model $M$ of) the model completion, the unary predicates are all infinite, and $R$ is a bipartite random graph between 
$Q^M_a$ and $P^M_b$ for $(a,b) \in \{ (0,0), (1,0), (1,1) \}$. Between $Q^M_0$ and $P^M_1$ there are no $R$-edges.  Notice something stronger:  
$R$ is also a bipartite random graph between $Q^M_1$ and $P^M_0 \cup P^M_1$. 
\end{itemize}

\br
Our theories $T_\xm$ can be informally thought of as descendants of these examples in the case where the unary predicates $Q_\eta, P_\nu$ 
are indexed by finitely branching trees of countable height, rather than trees of height $0$ or $1$, so the 
bipartite random graphs arise between type-definable ``leaves,''  
and we use background (``template'') sequences of sparse graphs to decide the pattern of 
where $R$ is forbidden.  Formally:

\begin{defn} Suppose we are given the following data. 
\begin{enumerate}
\item Let $\overline{m} = \langle m_n : n < \omega \rangle$ be a fast-growing sequence of natural numbers in the sense of \cite[Definition 6.1]{MiSh:1167}. 
\item Notation: let $\mct_1 = \mct_2 = (\{ \eta \in {^{\omega >} \omega} : \eta(n) < m_n \}, \tlf)$ be trees of countable height and finite branching; 
the branching at level $n$ is of size $m_n$.  
\item Let $\overline{E} = \langle E_n  : n < \omega \rangle$ be a sequence of sparse random\footnote{In this paper, finite random graphs are random in 
the usual sense of finite combinatorics; infinite random graphs are random in the sense of model theory.  It is also worth noting that the parameters for these theories involve graphs, although in the theories themselves, the edge relation is a bipartite graph. In translating from the graph to the theory, we simply double the vertices. See the next footnote.}
graphs satisfying \cite[Lemma 6.7]{MiSh:1167}:   
for each $n$, $E_n$ has $m_n$ vertices and for a reasonable notion of ``small'' and ``large,'' quantified there, 
every small set of vertices is contained in the neighborhood of a single vertex, and 
no large set of vertices is. 
\item Let $\xi: \omega \rightarrow \{ 0, 1 \}$ be a function which is $1$ infinitely often. 
\end{enumerate}
From this $\xm = \xm[\bar{m}, \bar{E}, \xi]$, called a parameter,  
we define the following universal theory $T^0_\xm$.  
The language has a binary relation $R$ and unary predicates 
\[ \{ \mcq, \mcp \} \cup 
{ \{ Q_\eta : \eta \in \mct_1 \} } \cup \{ P_\nu : \nu \in \mct_2 \}.  \] 
The axioms entail that: 
\begin{itemize}
\item $\mcq$, $\mcp$ partition the domain, $Q_{\langle \rangle} = \mcq$, 
 $P_{\langle \rangle} = \mcp$, and $R \subseteq \mcq \times \mcp$. 
 \item for $\eta \in \mct_1$ of height $n$, $\{ Q_{\eta^\smallfrown \langle \ell \rangle} : \ell < m_n \}$ partitions 
$Q_\eta$; \\ likewise, 
for $\nu \in \mct_2$ of height $n$, $\{ P_{\eta^\smallfrown \langle \ell \rangle} : \ell < m_n \}$ partitions  
$P_\eta$.  
\item Given $\eta \in \mct_1$ and $\nu \in \mct_2$, both of height $n$, there is an axiom forbidding 
$R$-edges between $Q_{\eta^\smallfrown \langle i \rangle}$ and $P_{\nu^\smallfrown \langle j \rangle}$ if the following two conditions both hold:
first, $(i,j)$ is not an edge\footnote{The graphs $E_n$ allow self-loops.} in $E_n$, and second, $\xi(n) = 1$. 
\end{itemize}
$T_\xm$ is the model completion of $T^0_\xm$. 
\end{defn}
The theories $T_\xm$ are well defined, eliminate quantifiers, and are 
simple rank one, with the only dividing coming from equality. (Essentially, everything not forbidden occurs.) This is proved in \cite{MiSh:1167} 2.21-2.22, though all of \S 2 of that paper, in particular 2.17-2.20, is devoted to 
developing the $T_\xm$.\footnote{We summarize here a complementary perspective on these theories, which the interested reader can find 
spelled out in \cite{MiSh:1167} \S\S 2-3. 
That is, the models of the theories $T_\xm$ can be informally seen as ``unions of random graphs,'' in the following sense. Define (for expository purposes; this is not 
definable or interpretable in $T_\xm$) a bipartite ``reduced graph'' 
whose vertices are $\lim(\mct_1)$, the leaves of $\mct_1$, on the left and $\lim(\mct_2)$, the leaves of $\mct_2$, on the right, 
with an edge between $\eta_*$ and $\rho_*$ if $R$-edges are not forbidden between $Q_{\eta_* \rstr n}$ and $P_{\rho_* \rstr n}$ for any $n<\omega$. 
Then in, say, any $\aleph_1$-saturated model $M$ of $T_\xm$, informally write ``$Q^M_{\eta_*}$'' for the type-definable set 
$\bigcap_{n<\omega} Q^M_{\eta_* \rstr_n}$, and likewise ``$P^M_{\nu_*}$'' for the type-definable set 
$\bigcap_{n<\omega} P^M_{\nu_* \rstr_n}$.  Then whenever $X \subseteq \lim(\mct_1)$, $Y \subseteq \lim(\mct_2)$ are sets of leaves, 
we have that in $M$, $R$ is a bipartite random graph between $\bigcup \{ Q^M_{\eta_*} : \eta_* \in X \}$ and $\bigcup \{ P^M_{\nu_*} : \nu_* \in Y \}$
when the bipartite reduced graph restricted to $X \times Y$ is complete, and is an empty graph when the bipartite reduced graph restricted 
to $X \times Y$ is empty. Putting these two facts together is enough to give the whole picture.} 

\br
\section{Preliminaries on ultrafilters} \label{s:prelim}

In this section we remind the reader of most definitions and earlier theorems we will need for our constructions below, though 
we will necessarily be brief and will focus on communicating the big picture.  

\subsubsection{Keisler's order} This is a pre-order on complete, countable theories defined by Keisler in 1967 \cite{keisler}, which can be 
thought of as ``comparing complexity''. It becomes a partial order modulo the equivalence relation of being in the same class. 

\begin{defn}[Keisler's order] \label{d:keisler}
Let $T_1, T_2$ be complete countable theories. 
Say $T_1 \tlf T_2$ if for any infinite cardinal $\lambda$, regular ultrafilter $\de$ on $\lambda$, and models 
$M_i$ of $T_i$, if the ultrapower $(M_2)^\lambda/\de$ is $\lambda^+$-saturated, then also the ultrapower 
$(M_1)^\lambda/\de$ is $\lambda^+$-saturated. 
\end{defn}

It is important to remember that because the ultrafilter is regular, the choice of model isn't important 
(within the elementary equivalence class), by a lemma of Keisler \cite{keisler}. That is, if $\de$ is a regular ultrafilter on $\lambda$, and $M, N$ are models of the same complete countable theory, 
then $M^\lambda/\de$ will be $\lambda^+$-saturated if and only if $N^\lambda/\de$ is $\lambda^+$-saturated. This is helpful because it says Keisler's order is 
really about theories and not about models, and also, it means 
we can work with whichever model of the theory may help our construction. 

It is also useful to know that by \cite{mm1}, Theorem 12, Keisler's order is local, meaning that if $M$ is a model of a complete countable theory, $\de$ is a regular ultrafilter on $\lambda$, and $M^\lambda/\de$ is not $\lambda^+$-saturated, then there is some formula $\vp$ so that $M^\lambda/\de$ is not $\lambda^+$-saturated for $\vp$-types.\footnote{i.e., types which involve just positive and negative instances of a single formula over some parameter set, rather than of all formulas in the language. Which formula this is may vary depending on the ultrapower.}

Some lengthier discussions of Keisler's order are in the original paper \cite{keisler}, in the second author's book \cite{Sh:c}, 
in the first author's thesis \cite{mm-thesis} and in the recent \cite{mm-icm}.  

\subsubsection{Boolean algebras} \label{sss:ba}
Next we will need some notation for Boolean algebras. 
Let ``$\ba^1_{\alpha, \mu, \theta}$'' mean the completion of a free Boolean algebra generated by 
$\alpha$ independent partitions, each of size $\mu$, where intersections of fewer than $\theta$ elements from distinct 
antichains are nonzero. 
More precisely, suppose we write the intended generators of this Boolean algebra as $\{ \mx_{\beta, \gamma} : \beta < \alpha, \gamma < \mu \}$. 
Let the axioms say that given any $u = \{ \mx_{\beta_i, \gamma_i} : i < i_* < \theta \}$, we have that $\bigcap u > 0$ if and only if 
$\beta_i = \beta_j \implies \gamma_i = \gamma_j$ for $i, j < i_*$. 
It will also be useful to have notation for such nonzero intersections of generators. Let 
``$f \in \fin_{\mu,\theta}(\alpha)$'' mean that $f$ is a function from $\alpha$ to $\mu$ whose domain has size $<\theta$. Let $\mx_f$ denote the 
(nonzero) intersection $\bigcap \{ \mx_{\beta, f(\beta)} : \beta \in \dom(f) \}$. Observe that elements of the form $\{ \mx_f : f \in \fin_{\mu, \theta}(\alpha) \}$ are 
dense in $\ba = \ba^1_{\alpha, \mu, \theta}$. 

It will be useful in what follows to identify such functions $f$ with their graphs, so that we can say things like ``$f \cup \{ (\gamma, i) \}$'' to indicate that we extend the function 
to take the value $i$ on input $\gamma$. 
In keeping with this notation, an expression like $\mx_{\{(\epsilon, 0)\}}$ will denote $\mx_f$ where $f$ is the function whose graph is $\{ (\epsilon, 0) \}$, which is none other 
than the generator called $\mx_{\epsilon, 0}$ in the previous paragraph. (Claim \ref{g20} below is a good example of why such notation is useful.)
For later reference, we summarize: 

\begin{concl} \label{ntn:ba}
The above discussion defines $\ba^1_{\alpha, \mu, \theta}$, $\fin$, and $\mx_f$. 
\end{concl}

\subsection{Goodness} \label{s:goodness}

Now we explain a combinatorial property of ultrafilters which is key for saturation. 
Step back for a moment to the basic setup of an ultrapower $M^I/\de$, where $|I| = \lambda$ (for clarity, we give $\lambda$ and the index set different names). 
Let $p(x) = \{ \vp_\alpha(x,a_\alpha) : \alpha < \lambda \}$ be a type in the ultrapower,\footnote{Since, as noted, Keisler's order is local, we could also assume all the $\vp_\alpha$'s are the same formula $\vp$, and just the $a_\alpha$'s differ.} 
and here the $a_\alpha$'s may be finite tuples.  Since $\de$ is a regular ultrafilter, we may fix some regularizing family $\{ X_\alpha : \alpha < \lambda \} \subseteq \de$, 
which recall means that the intersection of any infinitely many $X_\alpha$'s is empty. 

What it means for $p$ to be a type is, of course, by the fundamental theorem of ultrapowers, that 
any finitely many of its formulas have a common solution on some large (i.e., in $\de$) set of index models.  In particular, we can assign 
any finite $u \subseteq \lambda$ to a set in $\de$ as follows:  
\[ u \mapsto \{ t \in I : M \models \exists x \bigwedge_{\alpha \in u} \vp_\alpha(x,a_\alpha[t]) \} \cap \bigcap_{\alpha \in u} X_\alpha.  \] 
This map $f: [\lambda]^{<\aleph_0} \rightarrow \de$ is monotonic, meaning $u \subseteq v$ implies $f(v) \subseteq f(u)$, and its image is a regularizing family, meaning that 
for each $t \in \lambda$, the set $\{ \alpha < \lambda : t \in f(\{\alpha\}) \}$ of formulas assigned to index $t$ is finite.  It also has the property that 
$f(u) \cap f(v) \supseteq f(u \cup v)$. 
However, it is \emph{not necessarily} the case that equality holds. The key point, first pointed out by Keisler, is that the type $p$ is realized if and only if this $f$ has a multiplicative refinement, that is, if there is $g: [\lambda]^{<\aleph_0} \rightarrow \de$, so that for all finite $u, v \subseteq \lambda$ we have $g(u) \subseteq f(u)$, and $g(u) \cap g(v) = g(u \cup v)$.\footnote{For a detailed explanation of this point see e.g. \cite{mm-thesis} Chapter 1, Observation 2.} 
This motivates: 

\begin{defn}[Keisler] The filter or ultrafilter $\de$ on $\lambda$ is called \emph{$\lambda^+$-good}, or simply \emph{good}, if every $f: [\lambda]^{<\aleph_0} \rightarrow \de$ has a multiplicative refinement. 
\end{defn}

As discussed above, by work of Keisler, and Kunen in ZFC, good regular ultrafilters exist on every infinite cardinal. It will be useful to also say:

\begin{conv}
If $\de$ is a regular ultrafilter on $\lambda$ and $T$ a complete countable theory, say that $\de$ is ``good for $T$'' if for some, equivalently every, model $M$ of $T$ we have that $M^\lambda/\de$ is $\lambda^+$-saturated. 
\end{conv}

In this language, the above says that a good regular ultrafilter is good for any complete countable $T$. 

The point for Keisler's order is that, as Keisler saw, this all entails the order has a maximum class, because  
there 
are some theories which are `complicated enough' that they are saturated only by such ultrafilters.\footnote{Keisler proved this by finding complexity in the sense of being able to represent any failure of goodness as an omitted type, as in Peano arithmetic. It was subsequently surprising that linear order is maximal (Shelah \cite{Sh:c}, VI, \S 3) or indeed that $SOP_2$ is maximal (Malliaris-Shelah \cite{MiSh:998}).} So the maximum class consists of the nonempty class of theories $T$ with the property that the only ultrafilters which are good for $T$ are good.

Next we explain two innovations in ultrafilter construction which will be useful for the present paper. The first is due to \cite{MiSh:999} and the second 
to \cite{MiSh:1167}. 

\subsubsection{Separation of variables} 

The first idea we shall need is ``separation of variables,'' from \cite{MiSh:999}. As sketched in \S \ref{s:kunen} above, 
the state of the art following \cite{Sh:c} had been building regular ultrafilters on $\lambda$ or $I$, $|I| = \lambda$ 
by transfinite induction, making decisions about elements of $\mcp(\lambda)$, 
or about related elements of some independent family of functions.  There was a priori no place in this story for Boolean algebras other than $\mcp(\lambda)$, or 
for calling on the help of irregular ultrafilters. 

A different way one could try to proceed is the following; 
separation of variables essentially says that it will work. 
Start with $I$, $|I| = \lambda$, on which we want to build a regular ultrafilter. Choose a Boolean algebra $\ba$ which has size $2^\lambda$ and has maximal antichains of size $\leq \lambda$ (these are the only senses in which it needs to `look like' the power set of $\lambda$). Build a regular \emph{filter} $\de_0$ on $\lambda$ which is regular, which is 
``sufficiently expressive'' (more soon) and which has the property that the quotient $\mcp(I)/\de_0$ is our $\ba$.  Now the problem of building a regular ultrafilter on $I$ reduces to the problem of building an ultrafilter $\de_*$ on $\ba$, in the sense that given any $\de_*$ on $\ba$, we can ``combine'' $\de_0$ and $\de_*$ in the natural way to get an ultrafilter $\de$ on $\lambda$, which is regular because it extends $\de_0$.  Of course $\de_*$ itself need not be regular. 

The reason separation of variables is a theorem is twofold. First, since this is all in the service of Keisler's order, one has to be sure that moving the construction problem to the quotient won't irrevocably lose information: that is, we can ensure a given `goodness' problem from the original ultrapower is solved by transferring it to a related problem, called a `morality' problem (see next subsection) in $\ba$, $\de_*$ and solving this there. 
The theorem says that this will work if $\de_0$, the enveloping regular filter, meets a certain condition called excellence (for most including present purposes, good is enough), and goes on to say that for any $\ba$ as described, we can indeed arrange for the quotient to be isomorphic to $\ba$ while having 
$\de_0$ be excellent. 

Summarizing, separation of variables transfers the problem of building a regular ultrafilter on $|I|$, $|I| = \lambda$ onto 
any reasonable Boolean algebra (complete, of size $\leq 2^\lambda$, 
with the $\lambda^+$-c.c.).\footnote{For an even more detailed exposition, see also \cite{MiSh:1167} \S 7. }  
We will use the following notation. 

\begin{defn}[Regular ultrafilters built from tuples, from \cite{MiSh:999} Theorem 6.13] \label{d:built}
Suppose $\de$ is a regular ultrafilter on $I$, $|I| = \lambda$. We say that $\de$ is built from 
$(\de_0, \ba, \de_*, \jj)$ when: 

\begin{enumerate}
\item {$\de_0$ is a regular, $|I|^+$-good filter on $I$} 
\item {$\ba$ is a Boolean algebra}
\item {$\de_*$ is an ultrafilter on $\ba$}
\item {$\jj : \mcp(I) \rightarrow \ba$ is a surjective homomorphism such that:}
\begin{enumerate}
\item $\de_0 = \jj^{-1}(\{ 1_\ba \})$ 
\item $\de = \{ A \subseteq I : \jj(A) \in \de_* \}$.
\end{enumerate}
\end{enumerate}
\end{defn}

Finally, it will be useful to have some notation for the image under $\jj$ of sets which come from $\lost$ theorem. We have used $B \subseteq I$ and correspondingly 
$\mb \in \ba$ for sets on which certain formulas or sets of formulas have solutions, and $A \subseteq I$ and correspondingly $\ma \in \ba$ for sets on which things are true\footnote{this of course allows for some overlap, but in the present paper there shouldn't be confusion.}: 

\begin{ntn} \label{ntn:a}
Continuing in the notation of \ref{d:built}, suppose $\vp[\bar{a}]$ is a formula in which all free variables have been instantiated by parameters.  
Let $\ma[ ~\vp[\bar{a}]~ ] \in \ba$ denote $\jj ( ~ \{ t \in I :  ~M \models \vp[\bar{a}[t]] \} ~)$.   Below, we will use 
$\ma[ a = b ]$. 
\end{ntn}

\subsubsection{Morality} Before continuing we pause to explain ``morality.''   We have just seen that the problem of realizing a type amounts to solving a problem of goodness: whether a particular representation of the type $f: [\lambda]^{<\aleph_0} \rightarrow \de$ has a multiplicative refinement. We have also motivated transferring the problem of building ultrafilters on $\mcp(\lambda)$ onto some quotient $\ba$. When doing so the instances of goodness are necessarily also projected, and we need now to say what it means for them to be ``solved'' (that is, what kind of configuration in $\ba$ will ensure that back in $\de$, the instance of goodness is actually also solved).

\begin{defn} \label{d:enveloping}
Given 
$\ba$ and $\de_*$, imagine some ultrapower $M^I/\de$, $|I| = \lambda$, which arises by choosing  
$\de$ and $\jj$ which fulfill the requirements of \ref{d:built} for our given $\ba, \de_*$, and choosing some $M \models T$. 
Call such an ultrapower an ``enveloping ultrapower'' and of course, for a given $\ba$ there could be many such ``enveloping ultrapowers''. 
\end{defn} 

In $M^I/\de$, for any formula $\vp$ we will have $\vp$-types $\{ \vp(x,\bar{a}_\alpha) : \alpha < \lambda \rangle$ 
and can consider the ``\los map'' which for any finite subset $u$ of $\lambda$ sends 
$u \mapsto B_u := \{ t \in I : M \models \exists x \bigwedge_{\alpha \in u} \vp(x,\bar{a}_\alpha[t]) \}$, possibly intersected with 
$\bigcap_{\alpha \in u} X_\alpha$ where $\{ X_\alpha : \alpha < \lambda \}$ is some regularizing family. We explained above that this map has a multiplicative 
refinement (in the sense of the definition of good ultrafilter) if and only if the type is realized.

In \cite{MiSh:999} Definition 6.1 we gave a definition of ``possibility pattern'' that covered all sequences of the form $\bar{\mb} = \langle \mb_u : u \in [\lambda]^{<\aleph_0} \rangle$ 
in $\ba$ arising as images of some $\{ B_u : u \in [\lambda]^{<\aleph_0} \}$ which came specifically from a type in our given theory, in any enveloping ultrapower. 
That definition was slightly more general, but for the rest of the present paper, 
the reader can indeed safely take ``possibilty pattern'' to mean ``any $\bar{\mb}$ arising from some enveloping ultrapower in this way'' (and usually we will specify the $\vp$).   
The next definition simply states the analogue of good for the quotient: informally, that every possibility pattern has a multiplicative refinement. 

\begin{conv} \label{possibility}
Suppose we are given an ultrafilter $\de_*$ on a Boolean algebra $\ba$, a theory $T$, a formula $\vp$, and $\lambda$. 
\begin{enumerate}

\item[(a)] In this paper, say that $\bar{\mb} = \langle \mb_u : u \in [\lambda]^{<\aleph_0} \rangle$ a $(\lambda, T)$-possibility pattern to mean that it is the image of the sequence $\langle B_u : u \in [\lambda]^{<\aleph_0} \rangle$ representing some $\vp$-type $p = \{ \vp(x,a_\alpha) : \alpha < \lambda \}$ in some enveloping ultrapower $($see two paragraphs earlier$)$ of some model of $T$, for some formula $\vp$ of $T$.  When we want to fix $\vp$, say a $(\lambda, T, \vp)$-possibility pattern. 

\item[(b)] Say that $\de_*$  is ``moral'' for a theory $T$, or $(\lambda^+, T)$-moral, if 
whenever $\bar{\mb} = \langle \mb_u : u \in [\lambda]^{<\aleph_0} \rangle$  is a $(\lambda, T)$-possibility pattern, i.e. for any $\vp$, then \footnote{Sometimes dealing with certain $\vp$ will ensure all are handled: see the next subsection.}
$\bar{\mb}$ has a multiplicative refinement, meaning that there is $\bar{\mb}^\prime = \langle \mb^\prime_u : u \in [\lambda]^{<\aleph_0} \rangle$ such that 
$(i)$ for each $u$, $\mb^\prime_u \leq \mb_u$; $(ii)$ for each $u, v$, ~ $\mb^\prime_u \cap \mb^\prime_v = \mb^\prime_{u \cup v}$; and $(iii)$ each $\mb^\prime_u \in \de_*$. 
\end{enumerate}
 
\end{conv}

In this language, the background theorem of separation of variables, \cite{MiSh:999} Theorem 6.13, proves that if the ultrafilter $\de_*$ on $\ba$ is 
moral for a theory $T$ (recalling \ref{d:built}(1) that $\de_0$ is $|I|^+$-good)  then the corresponding regular ultrafilter $\de$ on $I$ 
is indeed good for $T$; and also the inverse, if $\de$ is good for $T$ then also $\jj(\de)$ is moral for $T$. 

\br

\subsubsection{Note on the locality of Keisler's order} \label{s:locality} As mentioned, Keisler's order is local (\cite{mm1} Theorem 12, meaning that 
in order to obtain saturation it suffices to realize $\vp$-types that is, maximal consistent sets of positive and negative instances of $\vp$ with 
parameters from some set of size $\leq \lambda$, for all formulas $\vp$). For the purposes of the present paper, it will be useful to remember that in some theories, 
we can restrict to a couple of distinguished formulas $\vp$ and know that realizing $\vp$-types for one of these ``distinguished'' formulas $\vp$ will suffice. 
For example, in linear order, it suffices to deal with $\vp(x,y) = x<y$, or $\vp(x;\bar{y}) = y_0 < x < y_1$, and in the random graph,  
it suffices to deal with $R(x,y)$, or with $\vp(x;\bar{y}) = R(x,y_0) \land \neg R(x,y_1)$, (where $R$ is the edge relation); this is immediate from the model theoretic point of view.  
In the case of $T_\xm$,  it suffices to deal with $Q_\nu(x) \land R(x,y)$ for each $\nu \in \mct_1$, 
provided that we ensure separately that the ultrafilter is good for the random graph; this is proved in 
\cite{MiSh:1167} Conclusion 5.6. 
We will use these facts in the proofs of Claim \ref{c:42} and Lemma \ref{cc-lemma}.

\br

\subsubsection{Building ultrafilters and Boolean algebras together by induction} The final idea we will need in the present paper was one of the 
main advances behind the proof in \cite{MiSh:1167}  
that Keisler's order has continuum many classes. 
At a high level, it addresses the following very interesting problem. Keisler's order talks about two kinds of ultrafilters on Boolean algebras -- the regular $\de$'s, and the types we are trying to realize. Can we bring these two closer together? 

The new idea was that at least for the theories $T_\xm$ (cf. Question \ref{q0.2} above),  we can start to build ultrafilters ``tailor-made'' for theories, working within the regime of separation of variables, by 
building the Boolean algebra $\ba$ and the ultrafilter $\de_*$ on it together, by induction.\footnote{Thus $\ba$ and $\de$ in the previous paragraph may be the final Boolean algebra and ultrafilter we have finished constructing, or some pair 
arising at some stage in an induction.} We do this by at each stage simply 
adding formal ``solutions'' [multiplicative refinements] to ``problems'' [possibility patterns] towards saturation for some theory, 
and the key point is that we add the solutions 
as freely as possible (modulo the constraints imposed by being a solution).   
Set theorists may observe a resonance with ideas from iterated forcing.  

The next notation simply says this mathematically; if the reader will pause to think what is needed in such a definition before reading it, they may 
find it natural, despite the abundance of notation.  Possibility pattern was defined in \ref{possibility} above.

\begin{defn}[\cite{MiSh:1167} Definition 10.12] \label{d:extn} 
Say that $\mfb = (\ba_\mfb, \de_\mfb)$ is a 
\[ \mbox{ $(\lambda, T, \vp)$-extension of $\mfa$ } \]
when there exists a $(\lambda, T, \vp)$-possibility pattern 
 $\bar{\mb} = \{ \mb_u : u \in [\lambda]^{<\aleph_0} \}$ such that $\mfb$ is a $(\lambda, T, \bar{\mb}, \vp)$-extension of $\mfa$, which means: 
\begin{enumerate}
\item $\ba_\mfb$ is the completion of the Boolean algebra $\ba$ generated by the set $\mcy_{\mfa, \mfb}$ which is 
$\ba_\mfa$ along with the set of new elements 
$\{ \mb^1_{\{\alpha\}} : \alpha < \lambda \}$, freely except for the set of equations $\Gamma_{\mfa, \mfb}$ which are:\footnote{
i.e. freely 
except for the rules already governing $\ba_\mfa$ and the new rules stating that $\bar{\mb}^1$ is 
a formal solution to $\bar{\mb}$. (The word ``equation'' here does not exclude inequalities.)}  
\begin{enumerate}
\item the equations already in $\ba_\mfa$.
\item for every nonempty finite $u \subseteq \lambda$, 
\[  \bigcap_{\alpha \in u} \mb^1_{ \{\alpha\} } \leq \mb_u. \]
\end{enumerate}
\item Notation: for $|u| > 1$, let $\mb^1_u := \bigcap_{\alpha \in u} \mb^1_{\{\alpha\}}$.  Convention: $\mb^1_\emptyset = 1_\ba$. 
\item 
When  $u = \{ \alpha \}$ and it is unlikely to cause confusion, we may drop parentheses and write $\mb^1_\alpha$ for $\mb^1_{\{ \alpha \}}$, 
so the new elements are $\{ \mb^1_\alpha : \alpha < \lambda \}$. 
\item $\de_\mfb$ is an ultrafilter on $\ba_\mfb$ which agrees with $\de_\mfa$ on $\ba_\mfa$, and such that $\mb^1_\alpha \in \de_\mfb$ 
for all $\alpha < \lambda$, \emph{if} such an ultrafilter exists; otherwise not defined. 
\end{enumerate}
\end{defn}

\br
\noindent The next fact, also natural upon reflection, simply uses that the generators are dense in the completion (and the $\Delta$-system lemma) to 
give a useful normal form for elements in such extensions.  

\begin{fact}[\cite{MiSh:1167} Observation 10.19] \label{smooth}
Suppose 
$\mfb$ is a $(\theta, T, \bar{\mb})$-extension of $\mfa$ for some $\theta \leq \lambda$. Let $\langle \ma^2_\alpha : \alpha < \kappa \rangle$ be a sequence of elements of $\ba^+_\mfb$, for some uncountable regular $\kappa$. Then: 

\begin{enumerate}
\item[$(1)$]  for each $\alpha < \kappa$, there is $\ii_\alpha = (\mx_\alpha, u_\alpha, \langle u_{\alpha, \ell} : \ell < n_\alpha \rangle)$  such that  $\mx_\alpha \in \ba^+_\mfa$; $n_\alpha \in \mathbb{N}$; $u_\alpha$, $u_{\alpha,0}, \dots, u_{\alpha,n_\alpha-1} 
\in [\theta]^{<\aleph_0}$; 
$u_\alpha \not\supseteq u_{\alpha,\ell}$ for $\ell < n_\alpha$; $\mx_\alpha \leq \mb_{u_\alpha}$; 
and 
\[ \ba_\mfb \models  0 < \mx_\alpha \cap \mb^1_{u_\alpha} \cap \bigcap_{\ell < n_\alpha} (- \mb^1_{u_{\alpha,\ell}})  \leq \ma^2_\alpha. \] 
\emph{\small{[i.e., since the generators are dense in the completion, we can find a positive element below $\ma^2_\alpha$ which is the intersection of 
an element from $\ba_\mfa$, some number of new elements, and some number of negations of (intersections of) new elements.]}}

\vspace{1mm}
\item[$(2)$] 
Given $\ii_\alpha$ for $\alpha < \kappa$ from $(1)$, define $w_\alpha = u_\alpha \cup \bigcup\{ u_{\alpha, \ell} : \ell < n_\alpha \}$.  
Then there are $\uu \in [\kappa]^\kappa$, 
$w_*$, $u_*$, $n_*$, ${\langle u^*_\ell : \ell < n_* \rangle}$ such that for every $\alpha \in \uu$, 
$w_* \subseteq w_\alpha$ and $\langle w_\alpha \setminus w_* : \alpha \in \uu \rangle$ are pairwise disjoint,  
$n_\alpha = n_*$, $u_\alpha \cap w_* = u_*$, $u_{\alpha, \ell} \cap w_* = u^*_\ell$.

\noindent \emph{\small{[i.e., by applying the $\Delta$-system lemma we can smooth this out on a large set.]}}
\br

\item[$(3)$] For every $\alpha \in \uu$ and $\mx_\alpha$ from $\ii_\alpha$, we have that $\mx_\alpha \bp \ma^2_\alpha$, \cite{MiSh:1167} $8.10$.

\br
\item[$(4)$] Suppose $\uu$ is from $(2)$ and $X \subseteq \uu$ is finite and $\ma_* \in \ba^+_\mfa$. Suppose
\[ \ba_\mfb \models \ma_* \cap  \bigcap_{\alpha \in X} \left(\mx_\alpha \cap \mb^1_{u_\alpha}\right)  > 0. \]
Then also 
\[ \ba_\mfb \models  \ma_* \cap \bigcap_{\alpha \in X}  \left(\mx_\alpha \cap \mb^1_{u_\alpha} \cap \bigcap_{\ell < n_\alpha} (- \mb^1_{u_{\alpha,\ell}}) \right)  > 0 \]
\emph{\small{[i.e., when checking for positive intersections we may safely ignore complements.]}}
\end{enumerate}
\end{fact}

This concludes our summary of ultrafilter construction. 

\subsubsection{Note on the pseudo-nfcp} \label{s:pnfcp} 
It is not necessary for the present proofs, but we take the opportunity to mention that there is an interesting and as yet fairly unexplored combinatorial property called the pseudo-nfcp which we defined in \cite{MiSh:1206}, Definition 4.1, and 
which holds, say, of all pairs $(T_\xm, \vp)$ for $\xm$ a parameter and $\vp(x,y) = R(x,y)$, as is proved in \cite{MiSh:1206}, Claim 3.4.  We mention it in light of 
Question \ref{q0.1}, since recall that the theories which are minimal in Keisler's order are precisely those with the nfcp, i.e., not the finite cover property. 
The reader may notice that some proofs involving saturation for the $T_\xm$'s below could instead be carried out for $(T, \vp)$ with the pseudo-nfcp. 

\br

We now turn to the proofs.

\vspace{5mm}

\section{First direction: saturation} \label{s:first-s}

In this section we fix $\bn > \bk \geq 2$ and we build by induction a Boolean algebra $\ba$ and an ultrafilter $\de_*$ on it, so that a corresponding regular ultrafilter $\de$ formed 
from them by separation of variables will be good for any $T_\xm$. \footnote{The reader is reminded that an overview of the subsequent proofs and how they fit together is given at the end of the first section.}

\begin{conv}
Fix for the section integers $\bn > \bk \geq 2$. 
\end{conv}

\begin{conv} \label{conv-a}
Fix for the section infinite cardinals $\del, \mu$ satisfying $\del > \mu$.  We use $\theta = \aleph_0$. 
\end{conv}

\begin{ntn} 
For us $\chi$ will always denote a regular uncountable cardinal. The definitions will be more interesting when $\chi > \del$.  
Let $(\mch(\chi); \in)$ denote the sets hereditarily of size $<\chi$ $($or: whose transitive closure has size $< \chi$$)$, so, a model of sufficient set theory 
$($ZFC minus power set$)$. 
\end{ntn}

The next definition describes a property of a family of small submodels, and notice it has two parts: individually they each satisfy 
certain closure conditions, and all together, they overlap only as expected. 
Note that \ref{d:n-k-pos}(2) implies $N_\emptyset \subseteq N_u$ for all $u \in [\bn]^{\leq \bk}$. 

\begin{defn} \label{d:n-k-pos} 
Let $\bM = (\mch(\chi); \in)$ or one of its expansions, assuming $\tau(\bM)$ 
is countable\footnote{We need at least this expressivity, and then we can have other things in the language too.}.
Say that the family of elementary submodels $\langle N_u : u \in [\bn]^{\leq \bk} \rangle$ is \emph{in $(\bn,\bk, \mu)$-position} in $\bM$, 
$($really $(\bn, \bk, \mu, \theta)$-position in $\bM$ and $\mu^{<\theta} = \mu$, but we can omit $\theta$ when, as here, it is $\aleph_0$$)$ when:  
\begin{enumerate}

\item for each $u \in [\bn]^{\leq \mk}$, we have:  
\\ $N_u \preceq \bM$, 
$||N_u|| = \mu$, $[N_u]^{<\theta} \subseteq N_u$, and $\mu+1 \subseteq N_u$, 
\item and for all  $u_1, u_2 \in [\bn]^{\leq \bk}$,
\[ N_{u_1} \cap N_{u_2} \supseteq N_{u_1 \cap u_2}.  \]
\end{enumerate}
\end{defn}

\noindent 

Note: the next definition assumes the $\mu^+$-c.c.; we will prove in \ref{c:42} that this is true of every $\ba_\alpha$ in our construction sequence. 

\begin{defn} \label{g2a} \emph{ } 
Let $\ba$ be a complete Boolean algebra with the $\mu^+$-c.c., and let $\bar{\mx} = \langle \mx_\epsilon : \epsilon < \del \rangle$ 
be a sequence of elements of $\ba^+$ which are independent in $\ba$, meaning that every finite boolean combination is nonzero. 

We say that $(\ba, \bar{\mx})$ satisfies the $(\bn, \bk, \mu)$-c.c. 
when (A) implies (B):
\begin{enumerate}

\item[(A)] \begin{enumerate}
\item $\ba, \del, \bar{\mx} \in \bM = (\mch(\chi); \in)$,  
\item $x \in \bM$ is an element coding\footnote{By this we simply mean that any model which contains this element necessarily contains all the 
members of the tuple. } 
$(\ba, \bar{\mx}, \mu, \theta)$, 
\item 
$\langle N_u : u \in [\bn]^{\leq \bk} \rangle$ is in $(\bn,\bk, \mu)$-position in $\bM$, with $x \in N_\emptyset$. 
\end{enumerate}

\sbr
\item[(B)]  
Choose any $\ma_u \in \ba \cap N_u$ for $u \in [\bn]^{\leq \bk}$. 
\emph{ If } $\ma := \bigcap_{u} \ma_u > 0_\ba$, then for every $n<\omega$, $\trv: n \rightarrow \{ 0, 1\}$ and distinct 
$\epsilon_0, \dots, \epsilon_{n-1} \in \del \setminus \bigcup \{ N_{u} :  u \in [n]^{\leq k} \}$  
we have that $\ma \cap \bigcap_{i<n} (\mx_{\epsilon_i})^{\trv(i)} > 0_\ba$.
\end{enumerate}
\end{defn}

\begin{disc} \label{d:3.2a}
\emph{We may think of $\ref{g2a}$ as describing a property of certain systems of small, sufficiently closed subalgebras of our $\ba$. 
Informally, $\ba$ contains some designated free sequence of size $\del$, and
whenever we take a suitable system of subalgebras of $\ba$ and an element from each, \underline{if} the intersection of these 
elements is nonempty, call it $\ma$, \underline{then} $\ma$ freely crosscuts elements of the designated sequence outside those 
subalgebras.   ``Suitability'' has 
two simple parts. First, our subalgebras have enough information, and second, they are small, of size $\mu$, and sufficiently closed.
We formalize this by saying that in some ambient model of sufficient set theory, $x$ codes $($say$)$ the finite tuple $\langle \ba, \bar{\mx}, \del, \mu, \theta \rangle$; 
we then choose elementary submodels $N_u$ for $u \subseteq k$ as described, all containing $x$ 
$($because $N_\emptyset \subseteq N_u$, so they will each have 
their own small versions of $\ba$, etc$)$; and we consider the subalgebras
arising as the intersection of $\ba$ with these elementary submodels.} 
\end{disc}

Next we prove that our chain condition \ref{g2a} holds naturally for free Boolean algebras. 
On the notation $ \mx_{\{ (\epsilon, 0) \}}$, recall \S \ref{sss:ba} above. 

\begin{claim} \label{g20}
Assume $\alpha \geq \del$, $\ba = \ba^1_{\alpha, \mu, \aleph_0}$ and let $\bar{\mx} = \langle \mx_\epsilon : \epsilon < \del \rangle$ where  
$\mx_\epsilon = \mx_{\{ (\epsilon, 0) \}}$ for $\epsilon < \del$.  Then $(\ba, \bar{\mx})$ satisfies the 
$(\bn, \bk, \mu)$-c.c.  
\end{claim}

\begin{proof}
Let $x$ be an element of $\bM$ which codes $(\ba, \bar{\mx}, \mu, \del, \theta)$ and let $\langle N_u : u \in [\bn]^{\leq \bk} \rangle$ 
be given, so this is a sequence of elementary submodels of $\bM$ which is in $(\bn, \bk, \mu)$-position and $x \in N_\emptyset$, thus 
$x \in N_u$ for each $u \in [\bn]^{\leq \bk}$. 
Suppose we are given $\langle \ma_u : u \in [\bn]^{\leq \bk} \rangle$ where each $\ma_u \in \ba^+ \cap N_u$ and $\ma^* = \bigcap_u \ma_u > 0_\ba$. 

Enumerate $[\bn]^{\leq \bk}$ as $\langle u_\ell : \ell < \bn^\bk \rangle$.  We shall choose $\ma^*_\ell$, $f_{u_\ell}$ by induction on 
$\ell \leq n$ such that $\ma^*_\ell \in \ba^+$, $\ma^*_{\ell+1} \leq \ma^*_\ell$, and $f_{u_\ell} \in N_{u_\ell}$ so that $\mx_{f_{u_\ell}} \cap \mx^*_\ell > 0$. 
Set $\ma^*_0 = \ma^*$.  Suppose $\ell < \bn^\bk$ and suppose $\ma^*_\ell$ has been defined. 
Working in $\bM$, we may choose in $\ba$, which remember means $\ba^{\bM}$, for each 
$u_\ell$ a maximal antichain $I_{u_\ell} = \langle \mx_f : f \in \fin_{\mu,\theta}(\alpha) \}$ supporting $\ma_{u_\ell}$, i.e., each member of the antichain is either 
$\leq \ma_{u_\ell}$ or disjoint to $\ma_{u_\ell}$.  Since $\ba$ has the $\mu^+$-c.c., this antichain will have cardinality $\leq \mu$, so in $\bM$ it may be without loss of generality  
enumerated by $\mu$. We are assuming $\mu + 1 \subseteq N_{u_\ell}$ and $\ma_{u_\ell} \in N_{u_\ell} \preceq \bM$ so we may assume $I_{u_\ell} \subseteq N_{u_\ell}$ 
where it remains a maximal antichain of $\ba^{N_{u_\ell}}$ supporting $\ma_{u_\ell}$.  
Since $I_{u_\ell}$ is a maximal antichain in $\ba$, working in $\bM$ we see that there is $f_{u_\ell}$ such that $\mx_{f_{u_\ell}} \in I_{u_\ell}$ and $\mx_{f_{u_\ell}} \cap \ma^*_\ell > 0$, 
hence $\mx_{f_{u_\ell}} \leq \ma_{u_\ell}$.
Since $I_{u_\ell} \subseteq N_{u_\ell}$, we also know $\mx_{f_{u_\ell}} \in N_{u_\ell}$.
Define $\ma^*_{\ell+1} := \ma^*_\ell \cap \mx_{f_{u_\ell}}$ in $\bM$. We can do this for each $\ell < \bn^\bk$ in turn. When we have finished, $f := \bigcup_\ell f_{u_\ell}$ is a function 
(since $\ba \models \bigcap_\ell \mx_{f_{u_\ell}} > 0$) and $\mx_{f} \leq \ma_u$ for each $\mu \in [\bn]^{\leq \bk}$, 
and moreover $\dom(f) \subseteq \bigcup \{ N_u : u \in [\bn]^{\leq \bk} \}$.  
So for any distinct 
$\epsilon_0, \dots, \epsilon_{n-1} \in \del \setminus \bigcup_u N_{u}$ and $\trv: n \rightarrow 2$,  
\[ f \cup \{ (\epsilon_i, \trv(i)) : i < n \} \in \fin_{\mu, \aleph_0} (\alpha) \] 
i.e. is a function, hence 
\[ 0 < \mx_f \cap \bigcap_{i < n} (\mx_{\epsilon_i})^{\trv(i)} \]
which is what we wanted to show. 
\end{proof}

\begin{conv}
For the rest of the section, we have in mind that 
\[ \lambda = \del \geq \mu^+ > \theta = \aleph_0. \] 
\end{conv}

We now turn to the construction of our Boolean algebras and ultrafilters on them by induction. 
Definition \ref{d:constr} explains the plan for our construction sequence.

\begin{defn} \label{d:constr}  \emph{ }

\begin{enumerate} 
\item[(A)] Let $\mathbf{p}$ contain:
\begin{enumerate}
\item[(1)] In general, a set of pairs $(T, \vp(\bar{x}, \bar{y}))$, where $T$ is a countable complete theory and $\vp$ is a formula of $T$. 
\\ In this section, the pairs we shall use are $(T_{rg}, \vp(x;\bar{y}) = R(x,y_0) \land \neg R(x,y_1))$, where $\trg$ is the theory of the random graph, or of the form 
$(T_\xm, Q_\nu(x) \land R(x,y))$ for one or more $\xm$'s. 
\item[(2)] a cardinal $\lambda$. 

\end{enumerate}
\item[(B)]
Say that $\bar{\mfa} = \langle \mfa_\alpha : \alpha < \alpha_* \rangle$ is a $\mathbf{p}$-\emph{construction sequence} when each 
$\mfa_\alpha$ is a pair $(\ba_\alpha, \de_\alpha)$ and these satisfy:
\begin{enumerate}
\item[(1)] Each $\ba_\alpha$ is a complete Boolean algebra and $\de_\alpha$ is an ultrafilter on $\ba_\alpha$. 
\item[(2)] $\ba_0 = \ba^1_{\del, \aleph_0, \aleph_0}$ and $\de_0$ is some fixed nonprincipal ultrafilter on $\ba_0$. 
\item[(3)] The sequence of Boolean algebras $\langle \ba_\alpha : \alpha < \alpha_* \rangle$ is increasing, and continuous
meaning that at limits we take the completion of the union.  Moreover, each 
$\ba_{\alpha} \lessdot \ba_{\alpha+1}$. 
\item[(4)] The sequence $\langle \de_\alpha : \alpha < \alpha_* \rangle$ of ultrafilters is increasing, and continuous \\ meaning that 
for limit $\gamma$, $\de_\gamma$ includes $\bigcup_{\beta < \gamma} \de_\beta$, if such an ultrafilter exists. 
\item[(5)] For each successor stage $\alpha = \beta+1$, for $\lambda$ and some pair $(T, \vp)$ from $\mathbf{p}$, 
$(\ba_\alpha, \de_\alpha)$ is a $(\lambda, T, \vp)$-extension of $(\ba_\beta, \de_\beta)$,
recalling $\ref{d:extn}$ above.
\end{enumerate}
\end{enumerate}
\end{defn}

\begin{conv}
In this section, again, $\mathbf{p}$ should contain the pair $(\trg, \vp(x,\bar{y}))$ where $\vp(x,\bar{y}) = R(x,y_0) \land \neg R(x,y_1))$, 
$\trg$ is the theory of the random graph and $R$ is the edge relation; and 
any or all pairs of the form $(T_\xm, Q_\nu(x) \land R(x,y))$ ranging over any set of parameters $\xm$. 
\end{conv}

Since the inductions will involve taking basic 
extensions in the sense of Defintion \ref{d:extn} above, we start by recalling that it has been verified that these extensions generally behave 
well, even before calling in properties of the theory. Fact \ref{f:1017} recalls that the successor step works, and 
Fact \ref{f:rmk} recalls that the limit step works. (However, we will check Fact \ref{f:rmk}(b) by hand below, since it  
will follow from stronger conditions we need to prove.)   

\begin{fact}[\cite{MiSh:1167} Claim 10.17] \label{f:1017} 
Suppose $(\ba_{\alpha+1}, \de_{\alpha+1})$ is a $(\lambda, T, \bar{\mb})$-extension of $(\ba_\alpha, \de_\alpha)$, for some theory $T$. Then $\ba_\alpha \subseteq \ba_{\alpha+1}$, 
indeed it is a complete subalgebra, in symbols  
 $\ba_\alpha \lessdot \ba_{\alpha+1}$.  Also, there exists an ultrafilter $\de$ on $\ba_{\alpha+1}$ 
 agreeing with  
$\de_\alpha$ on $\ba_\alpha$ and containing $\ba^1_\alpha$, hence $\de_{\alpha+1}$ is such an ultrafilter. 
\end{fact}

\begin{fact}[see \cite{Jech} Lemma 30.25] \label{proj-fact}
If $A$ is a complete subalgebra of a complete Boolean algebra $B$, thus a regular subalgebra of $B$, 
then for every $\mb \in B^+$ there exists $\ma \in A^+$ such that for every $\mx \in A^+$, if 
$0 < \mx \leq \ma$ in $A$ then $\mx \cap \mb > 0$ in $B$.
\end{fact}

\begin{fact} \label{f:rmk}  Not yet using any special properties of the theories, 
it follows from the basic properties of a construction sequence and Fact \ref{f:1017} that:
\begin{enumerate}
\item[(a)] $\ba_\alpha \subseteq \ba_\beta$ and indeed $\ba_\alpha \lessdot \ba_\beta$  for any $\alpha \leq \beta < \alpha_*$. 

\item[(b)] if $\alpha_*$ is a limit and each $\ba_\alpha$ satisfies the $\kappa$-c.c. for $\alpha < \alpha_*$, then 
also $\ba_{\alpha_*}$ satisfies the $\kappa$-c.c., see for instance \cite{Jech} Corollary 16.10.\footnote{This quotation is to give context; in our case, it will 
follow directly in the proof of \ref{c:42} below.}

\item[(c)]  As long as the enumeration of problems of some ~$~T$ ensures that each relevant problem is handled at some successor stage, 
the ultrafilter $\de_*$ on $\ba_* = \ba_{\alpha_*}$ will be moral for $T$, 
meaning that it will solve all its possibility patterns $\langle \mb_u : u \in [\lambda]^{<\aleph_0} \rangle$ $($because of the 
last clause of $\ref{f:1017}$$)$.  {A detailed account of the bookkeeping relevant here appears in the proof of \ref{c:constr} below.} 
\end{enumerate}
\end{fact}

What we have not yet established, of course, is whether we can carry out the construction of the sequence while keeping the 
Boolean algebras fairly constrained, as measured by their corresponding ability to \emph{not} solve problems for certain theories outside $\mct$. 
Here we will need to use specific properties of the theories in $\mct$.  
In the next few claims we establish that for our chosen $\mct$,  
elements of a construction sequence are as desired:  
the $\ba$'s have the $\mu^+$-c.c. and indeed are 
$(\mu^+, m)$-Knaster for any finite $m$ (\ref{c:42}), 
and the $\ba$'s satisfy the $(\bn, \bk, \mu)$-c.c.  (\ref{cc-lemma}).  (That these constraints suffice for omitting 
other types will be proved in \S \ref{s:n-s}.) Recall that: 

\begin{defn} \label{d:knaster}
Say that the Boolean algebra $\ba$ satisfies the $(\sigma, n)$-Knaster condition when: given $\ma_\epsilon \in \ba^+$ for $\epsilon < \sigma$, 
there is $\uu \in [\sigma]^\sigma$ such that if $u \subseteq \uu$, $|u| < 1 + n $ then $\bigcap \{ \ma_\epsilon : \epsilon \in u \} > 0$. 
\end{defn}

\begin{claim} \label{c:42}
Let $\bar{\mfa} = \langle \mfa_\alpha : \alpha < \alpha_* \rangle$ be a construction sequence. For every $\alpha < \alpha_*$, 
\begin{enumerate}
\item $\ba_\alpha$ satisfies the $\mu^+$-c.c. 
\item $\ba_\alpha$ satisfies the $(\mu^+, m)$-Knaster condition for every $m<\omega$, see $\ref{d:knaster}$.  
\end{enumerate}
\end{claim}

\begin{proof}
For Claim \ref{c:42}, it will suffice to prove that each $\ba_\alpha$ satisfies the Knaster condition since this a fortiori implies the $\mu^+$-c.c. 
The proof is by induction on $\alpha$. We will split the proof into two parts for easier reading: limit stages (including zero) and successor stages.

We shall use freely that for any ordinal $\alpha$, in particular, 
\[ (\star) \hspace{10mm} \mbox{for any limit $\alpha$, } \bigcup \{ \ba_\beta : \beta < \alpha \} \mbox{ is dense in } \ba_\alpha. \] 
So if we are given $\ma \in \ba_\alpha$, in most cases without loss of generality we can assume $\ma \in \bigcup \{ \ba_\beta : \beta < \alpha \}$. 

\begin{proof}[Proof of Claim \ref{c:42} for $\alpha = 0$ or $\alpha$ a nonzero limit.]  
For $\alpha = 0$, recall that in defining a construction sequence, $\ba_0 = \ba^1_{\del, \aleph_0, \aleph_0}$ 
is the completion of a free Boolean algebra with enough antichains.  Suppose then that we are given $\langle \ma_\epsilon : \epsilon < \mu^+ \rangle$
and $m<\omega$. 
For each $\ma_\epsilon$ we may choose 
$f_\epsilon \in \fin_{\aleph_0, \aleph_0}(\alpha_0)$ such that $\mx_{f_\epsilon} \leq \ma_\epsilon$. Let $u_\epsilon = \dom(f_\epsilon) 
\in [\alpha_0]^{<\aleph_0}$. By the $\Delta$-system lemma, for some finite $u_*$ and some $X \in [\mu^+]^{\mu^+}$, 
we have $\epsilon \neq \delta \in X$ implies $u_\epsilon \cap u_\delta = u_*$.  As the range of each $f_\epsilon$ is a subset of 
$\aleph_0$, we can further restrict to $\uu \in [X]^{\mu^+}$ so that $\epsilon, \delta \in \uu$ implies 
$f_\epsilon \rstr u_* = f_\delta \rstr u_*$.  Then for any finite $u \subseteq \uu$, we have that 
$\bigcup \{ f_\epsilon : \epsilon \in u \}$ is a function, thus $\bigcap \{ \mx_{f_\epsilon} : \epsilon \in u \}$ 
is nonzero, thus $\bigcap \{ \ma_{\alpha_\epsilon} : \epsilon \in u \}$ is nonzero, which is stronger than $(\mu^+, m)$-Knaster 
as we did not require $|u| < 1+m$. 

\br 
\noindent \emph{For $\alpha$ limit of cofinality $\neq \mu^+$}: Suppose we are given 
$\langle \ma_\epsilon : \epsilon < \mu^+ \rangle$.   By $(\star)$, 
without loss of generality for each $\epsilon < \mu^+$, $\ma_{\epsilon} \in \bigcup \{ \ba_\beta : \beta < \alpha \}$, hence 
there is $\gamma < \alpha$ 
such that $\{ \epsilon < \mu^+ : \ma_\epsilon \in \ba_\gamma \}$ has size $\mu^+$, so we may 
apply the inductive hypothesis. 

\br
\noindent \emph{For $\alpha$ limit of cofinality $\mu^+$:}  
Suppose that we are given a sequence $\bar{\mc} = \langle \mc_\gamma : \gamma < \mu^+ \rangle$ of elements of $\ba^+_\alpha$. 
Fix a strictly increasing and continuous sequence of ordinals $\bar{\iota} = \langle i_\gamma : \gamma < \mu^+ \rangle$ whose limit is $\alpha$.  
By $(\star)$, without loss of generality each $\mc_\gamma \in \bigcup_{\gamma < \mu^+} \ba^+_{i_\gamma}$. 
So for each $\gamma < \mu^+$, there is $\zeta(\gamma) \in (\gamma, \mu^+)$ such that $\mc_\gamma \in \ba_{i_{\zeta(\gamma)}}$. As $\gamma < \zeta(\gamma)$, 
we know (\ref{f:1017}) that $\ba_{i_\gamma} \lessdot \ba_{i_{\zeta(\gamma)}}$ and so (\ref{proj-fact}) there is 
$\mb_\gamma \in \ba^+_{i_\gamma}$ such that for any $\mx \in \ba_{i_\gamma}$, if $\ba_{i_\gamma} \models 0 < \mx \leq \mb_\gamma$, then 
$\ba_{i_{\zeta(\gamma)}} \models \mx \cap \mc_\gamma > 0$. 

 [We can say $\mb_\gamma$ is a projection of $\mc_\gamma$ to $\ba_{i_\gamma}$.]

Let $\langle \mb_\gamma : \gamma < \mu^+ \rangle$ be the sequence of elements defined in this way. 
Since we chose $\bar{\iota}$ to be increasing and continuous, for every limit $\gamma < \mu^+$, $\bigcup_{\epsilon < \gamma} \ba_{i_\epsilon}$ is dense in $\ba_{i_\gamma}$. 
So for every limit $\gamma < \mu^+$, as $\mb_\gamma \in \ba_{i_\gamma}$ we can choose $\ma_\gamma \in \bigcup_{\epsilon < \gamma} \ba^+_{i_\gamma}$ such that 
$\ba_{i_\gamma} \models 0 < \ma_\gamma \leq \mb_\gamma$, and $\ma_\gamma \in \ba^+_{i_{f(\gamma)}}$ for some $f(\gamma) < \gamma$. 

The function $\gamma \mapsto f(\gamma)$ is defined and regressive on the limit ordinals $\gamma < \mu^+$, so by Fodor there is $\gamma_* < \mu^+$ so that 
$\vv = \{ \gamma < \mu^+ : \gamma $ is a limit and $f(\gamma) < \gamma_* \}$ is stationary.  Let $E$ be a closed unbounded subset of $\mu^+$ with the property that 
$\gamma < \gamma^\prime \in E$ implies $\zeta(\gamma) < \gamma^\prime$.  Finally, let $\uu = \vv \cap E$, also stationary of size $\mu^+$.  

We claim it suffices to apply the inductive hypothesis to $\langle \ma_\gamma : \gamma \in \uu \rangle$ in $\ba_{i_{\gamma_*}}$. Why? Let us verify by induction on $k$ 
that for any finite $k$, 
and for any $\gamma_0 <  \cdots < \gamma_{k-1}$ from $\uu$, and any $\mx \in \ba^+_{i_{\gamma_*}}$, 
if $\bigcap \{ \ma_{\gamma_0}, \dots, \ma_{\gamma_{k-1}}\} \geq \mx > 0$ in $\ba_{i_{\gamma_*}}$ then 
 $\bigcap \{ \mc_{\gamma_0}, \dots, \mc_{\gamma_{k-1}} \} \cap \mx > 0$ in $\ba_{\alpha}$. 
When $k = 1$, suppose we are given $\gamma \in \uu$ and $0 < \mx \leq \ma_\gamma$ in $\ba_{i_{\gamma_*}}$. 
By construction $\mx \leq \ma_\gamma \leq \mb_\gamma$ in $\ba_{i_\gamma}$, so by definition of 
$\mb_\gamma$ as a projection, we have that $\mx \cap \mc_\gamma > 0$ in $\ba_{\alpha}$, as desired.   Suppose then that $k>1$ and 
we are given $\gamma_0 < \cdots < \gamma_k$ from $\uu$ and $\mx \in \ba^+_{i_{\gamma_*}}$ so that 
$\ma_{\gamma_0} \cap  \cdots \cap \ma_{\gamma_{k-1}} \geq \mx > 0$ in $\ba_{i_{\gamma_*}}$. By inductive hypothesis, 
 $\mx^\prime := \mx \cap \mc_{\gamma_0} \cap \cdots \cap \mc_{\gamma_{k-1}}$ is $>0$ in $\ba_{\alpha}$. 
 Letting $j = \max \{ \gamma_*, \zeta(\gamma_{k}) \}$, all of $\mx,  \mc_{\gamma_0}, \dots, \mc_{\gamma_{k}}$ belong to $\ba_{i_{j}}$, 
 because of the definition of $E$.  So in $\ba_{i_j}$, $\mx^\prime$ is defined and positive.  Also as $j \geq \gamma_*$, 
 $\ba_{i_j} \models 0 < \mx^\prime \leq \ma_{k} \leq \mb_k$. So by definition of $\mb_k$ as a projection, $\mx^\prime \cap \mc_k > 0$ in $\ba_\alpha$. 
We conclude by definition of $\mx^\prime$ that in $\ba_\alpha$, $\mx^\prime \cap \mc_{\gamma_0} \cap \cdots \cap \mc_{\gamma_{k}} > 0$, and as 
$\mx^\prime \leq \mx$ this suffices.
\end{proof}

Before continuing, we note for the interested reader that the case of $\alpha$ limit of cofinality $\mu^+$ still looks 
quite like \cite{MiSh:1167} 8.18.  Now things start to diverge a bit.  

\begin{proof}[Proof of Claim \ref{c:42} for $\alpha = \beta + 1$.]
There will be two cases depending on whether $T$ is $\trg$ (the theory of the random graph) or $T_\xm$, but they have a common beginning.  

Suppose we are given $m<\omega$ and $\langle \ma_\epsilon : \epsilon < \mu^+ \rangle$ a sequence of positive elements of $\ba_\beta$. 
By the normal form lemma \ref{smooth}, for each $\epsilon$ there are $u_\epsilon \in [\mu^+]^{<\aleph_0}$ and 
$\mx_\epsilon \leq \mb_{u_\epsilon}$ in $\ba^+_\beta$, and 
\[ \ba_\alpha \models  0 < \mx_\epsilon \cap \mb^1_{u_\epsilon} \cap \bigcap_{\ell < n_\epsilon} (- \mb^1_{u_{\epsilon,\ell}})  \leq \ma_\epsilon. \] 
Let $u^+_\epsilon = \bigcup \{ u_{\epsilon, \ell} : \ell < n_\epsilon \} \cup u_\epsilon$. 
Moreover, as in \ref{smooth}, we may move to a subset $\uu_0 \subseteq \mu^+$ of size $\mu^+$ where the 
$u^+_\epsilon$'s all have the same integer size and form a $\Delta$-system with root (or heart) $u_*$, 
and also $\langle u_\epsilon \cap u_* : \epsilon \in \uu_0 \rangle$ is constantly $u_{**}$. Now as in the proof of \ref{smooth}(2), (4), 
we may without loss of generality ignore the negative terms. 
So it will suffice to show that some subset of 
$  \{ \mx_\epsilon \cap \mb^1_{u_\epsilon}  : \epsilon \in \uu_0 \} $ 
witnesses the $(\mu^+, m)$-Knaster condition. 
Also, since the elements $\mx_\epsilon$ all belong to $\ba_\alpha$, by 
inductive hypothesis Knaster applies to $\ba_\alpha$ so we may move to $\uu_1 \subseteq \uu_0$, $|\uu_1| = \mu^+$ 
such that any $m$ elements of 
\[ \{ \mx_\epsilon : \epsilon \in \uu_1 \} \]
have nonempty intersection. So it suffices to show that 
\[  \{ \mx_\epsilon \cap \mb^1_{u_\epsilon}  : \epsilon \in \uu_1 \} \]
witnesses the $(\mu^+, m)$-Knaster condition.  For \emph{this} it suffices to show that for some 
$\uu_2 \in [\uu_1]^{\mu^+}$, for any $U \subseteq \uu_2$ with $|U| \leq m$, 
\[  \bigcap \{ \mx_\epsilon  : \epsilon \in U \} \leq \mb_W
\mbox{ where $W = \bigcup_{\epsilon \in U } u_\epsilon$ }.  \] 
Recall that characteristic sequences are monotonic in the sense that $v \subseteq u$ implies $\mb_u \leq \mb_v$. 
(So the reason this is sufficient is that there is then a homomorphism 
from $\ba_{\beta}$ to $\ba_\alpha$ which is the identity on $\ba_\alpha$ 
and sends each $\mb^1_{u_\epsilon}$ ($\epsilon \in U$) to $\mb_W$. Under this homomorphism, 
the image of the intersection $\{ \mx_\epsilon \cap \mb^1_{u_\epsilon} : \epsilon \in U \}$ would reduce to the 
intersection of the corresponding $\mx_\epsilon$'s so in particular would be nonzero.)

At this point, there are two cases depending on the kind of theory. Recall from \S \ref{s:locality} above that dealing with the formulas mentioned 
in each case will suffice.\footnote{We will not use this here, but note that the case where $(T, \vp)$ has the pseudo-nfcp in the sense of \S \ref{s:pnfcp} above would be similar 
to Case 1.}

\br
\noindent \underline{Case 1.} Assume $T = T_\xm$ and $\vp(x,y) = Q_\nu(x) \land R(x,y)$. 
Let $s$ be the common size of $|u_\epsilon|$ for $\epsilon \in \uu_1$ and recall $m$ was given at the beginning of the proof (in the third line of $\alpha = \beta + 1$). 
It follows from the 
axioms for $T_\xm$ that there is some finite $M = M(s,m)$ such that:
\begin{quotation}
\noindent $(\star)$ Suppose we have $m$ subsets of $\lambda$, $u_0, \dots, u_{m-1}$ each of size $s$ 
and for each $i<m$, $u_i = \{ \gamma_{i,\ell} : \ell < s \}$ and 
$\{ R(x,a_{\gamma_{i,\ell}}) : \ell < s \}$ is a partial type in the 
monster model for $T_\xm$.  For each $i < m$, let $\bar{p}_i = \langle p_{i,\ell} : \ell < s \rangle$ 
where $p_{i,\ell}$ is the quantifier-free type of $a_{\gamma_{i,\ell}}$ in the language 
$\{ P_\eta : \eta \in \mct_2, |\eta| \leq M \}$.  Suppose that the sequence $\langle \bar{p}_i : i < m \rangle$ 
is constant.  Then 
\[ \{ Q_\nu(x) \land R(x,a_{\gamma_{i,\ell}}) : i < m, \ell < s \} \] 
is also a partial type. 
\end{quotation} 
Let $M = M(s,m)$. 
Since $\bar{\mb}$ was a possibility pattern for $(T_\xm, \vp)$, 
$\mb_{u}$ represents a set of formulas of the form $\{ Q_\nu(x) \land R(x,a_\gamma) : \gamma \in u \}$. 
So for each $\epsilon \in \uu_1$, without loss of generality (since we could have replaced each $\mx_\epsilon$ by something smaller and still positive 
before moving from $\uu_0$ to $\uu_1$)  
$\mx_\epsilon$ decides $\ma[\vt(a_\gamma)]$ for $\gamma \in u_\epsilon$ and $\vt$ ranging over the predicates from the first $M$ levels of $\mct_2$.
Fix an enumeration of each $u_\epsilon$ in which all elements of $u_* \cap u_\epsilon$ come before all elements of 
$u_\epsilon \setminus u_*$. Let $\uu_2 \subseteq \uu_1$, $|\uu_2| = \mu^+$ be such that the sequence  
$\langle p_{\epsilon, \ell} : \ell < s \rangle$ is constant for $\epsilon \in \uu_2$.
Recall that the normal form \ref{smooth}(1) ensures (already in $\ba_\alpha$) that each 
$\mx_\epsilon \leq \mb_{u_\epsilon}$. Thus on $\mx_\epsilon$, 
each $\{ Q_\nu(x) \land R(x,a_\gamma) : \gamma \in u_\epsilon \}$ is a partial type. Now $(\star)$ and the choice of $M$ 
tell us that for any $U \subseteq \uu_2$, 
$|U| \leq m$, the set 
\[ \{ \{ Q_\nu(x) \land R(x,a_{\gamma}) : \gamma \in u_\epsilon \} : \epsilon \in U \} \]
is also a partial type, thus (by definition of possibility pattern)
\[  \bigcap \{ \mx_\epsilon  : \epsilon \in U \} \leq \mb_w
\mbox{ where $W = \bigcup_{\epsilon \in U } u_\epsilon$ }  \] 
which is, as noted above, sufficient. 

\br
\noindent \underline{Case 2.} 
Assume $T_\beta = \trg$, the theory of the random graph, and $\vp_\beta(x;y,z) = R(x,y_1) \land \neg R(x,y_0)$ where $R$ is the edge relation. 
So our possibility pattern reflects a background type of the form 
\[ \{ R(x,a_{(\gamma,1)}) \land \neg R(x,a_{(\gamma,0)}) : \gamma < \lambda \}. \] 
It will be useful to have a well-ordering of all parameters mentioned in the type, so use the lexicographic ordering on 
$\{ (\gamma, \ii) : \gamma < \lambda, \ii < 2 \}$. (This is just first coordinate, then second coordinate, so interpolates the positive and 
negative -- as in \cite{MiSh:1009}.) 
Now for every $\gamma \in u_\epsilon$ there is some $\mx$ with $0 < \mx \leq \mx_\epsilon$
and $\mx$ ``decides equality'' for $(\gamma, \ii)$, which means that there is $(\gamma^\prime, \ii^\prime) \leq_{\lex} (\gamma, \ii)$ such that 
\[  \mx \leq \ma[ a_{(\gamma, \ii)} = a_{(\gamma^\prime, \ii^\prime)}] \mbox{ and for all } (\gamma^{\prime\prime}, \ii^{\prime\prime}) <_{\lex} (\gamma^\prime, \ii^\prime), ~~  \mx \cap \ma[ a_{(\gamma, \ii)} = a_{(\gamma^\prime, \ii^\prime)}] = 0. \]
(This just uses the well ordering and the observation that $\mx_\epsilon \leq \ma[ a_{(\gamma, \ii)} = a_{(\gamma, \ii)}]$.)  
Since each $u_\epsilon$ is finite, we can assume, 
without loss of generality, that each $\mx_\epsilon$ decides equality for its set of indices $\{ (\gamma, \ii) : \gamma \in u_\epsilon, \ii < 2 \}$. For each 
$\epsilon < \mu^+$, define $f_\epsilon$ to be the function with domain $\{ (\gamma, \ii) : \gamma \in u_\epsilon, \ii < 2 \}$ and range 
$\{ (\gamma, \ii) : \gamma < \lambda, \ii < 2 \}$ which ``records the value of the collapse,'' that is, $(\gamma, \ii) \mapsto (\gamma^\prime, \ii^\prime)$ 
for the unique value satisfying the equation above for $\mx = \mx_\epsilon$. 

Now we use the $\Delta$-system lemma to smooth out the sets to which the positive and negative values collapse. That is, 
for each $\epsilon < \mu^+$ let $v_{\epsilon, 0} = \{ f_\epsilon(\gamma, 0) : \gamma \in u_\epsilon \}$ and let 
$v_{\epsilon, 1} = \{ f_\epsilon(\gamma, 1), : \gamma \in u_\epsilon \}$.  Choose $\uu_2 \in [\uu_1]^{\mu^+}$ so that 
$\langle v_{\epsilon, 1} : \gamma < \mu^+ \rangle$ and $\langle v_{\epsilon, 0} : \epsilon < \mu+ \rangle$ are each $\Delta$-systems, with hearts 
$v^*_1$ and $v^*_0$ respectively. 
Since $\mx_\epsilon \leq \mb_{u_\epsilon}$, necessarily 
$v_{\epsilon, 0} \cap v_{\epsilon_1} = \emptyset$ for each $\epsilon$, so also $v^*_1 \cap v^*_0 = \emptyset$.   
Now for any $\epsilon, \epsilon^\prime \in \uu_2$, without loss of generality 
we have that 
$( v^0_\epsilon \cup v^0_{\epsilon^\prime} ) \cap  ( v^1_\epsilon \cup v^1_{\epsilon^\prime}) = \emptyset$: thus, 
\[ \mx_{\epsilon} \cap \mx_{\epsilon^\prime} \leq \mb_w \mbox{ where } w = u_\epsilon \cup u_{\epsilon^\prime}. \]
Since, in the random graph, consistency of a set of instances of $\vp$ follows from the consistency of every two such instances, this implies that 
for any finite $U \subseteq \uu_2$ (not just those of size $m$),  
\[ \bigcap_{\epsilon \in U} \mx_\epsilon \leq \mb_W \mbox{ where $W = \bigcup_{\epsilon \in U} u_\epsilon$ } \]
which completes the proof. 
\end{proof}

\noindent This completes the proof of Claim \ref{c:42}. \end{proof}

\begin{lemma} \label{cc-lemma}
Let $\bar{\mfa} = \langle \mfa_\alpha : \alpha < \alpha_* \rangle$ be a construction sequence. Then for every $\alpha \leq \alpha_*$ and every 
$\bk, \bn$ with $\bk < \bn < \omega$, our 
$(\ba_\alpha, \bar{\mx})$ has the $(\bn, \bk, \mu)$-c.c. when $\bar{\mx} = \langle \mx_\epsilon = 
\mx_{\{ (\epsilon, 0) \}} : \epsilon < \del \rangle$ from $\ba_0$. 
\end{lemma}

\begin{proof} 
Fix $\bn, \bk$ for the proof.  
Since $\alpha_*$, the length of the sequence, was arbitrary and initial segments of construction sequences are construction sequences, 
it will suffice to prove this for $\ba_{\alpha_*}$, and note it holds for $\alpha_* = 0$ by \ref{g20}. 
 
Normally we might hope to work up to $\alpha_*$ by  
induction on $\alpha$, but note here a delicate point. 
In the present c.c. we make use of various elementary submodels where all of the prior ordinals may not appear, 
say, $\beta < \alpha$ need not imply $\beta \in \bigcup_u N_u$. So instead, at stage $\alpha_*$, 
we work by induction on a related well-order which will do the job.  (Essentially, we re-run an entire proof by induction using only ordinals 
which appear, and \emph{then} conclude the property for $\alpha_*$.)

Suppose then that $\chi$ is sufficiently large and $\bM = (\mch(\chi), \in)$ contains $\ba_{\alpha_*}$ and suppose that 
 $\bar{N} = \langle N_u : u \in [\bn]^{\leq \bk} \rangle$, $x$ are given as in \ref{g2a}(A).

First we define the well-ordering for the induction. 
Fix in advance some background linear ordering on $\{ u : u \in [\bn]^{\leq \bk} \}$, which will be used in the third bullet point below 
so that the lexicographic comparison is well defined. 
Let: 
\[ \mcs = \{ \bar{\beta} : \bar{\beta} = \langle \beta_u : u \in  [\bn]^{\leq \bk} \rangle \mbox{ and } \beta_u \in (\alpha_* + 1) \cap N_{u} \}. \]
Any such $\bar{\beta}$ represents the choice of a relevant ordinal from each $N_u$; 
notice that the ordinals in any such $\bar{\beta}$ need not be in any order of size 
and may repeat.  
 
Define $\leq_\mcs$ to be the following order on $\mcs$:

\begin{quotation}
\noindent $\bar{\beta} \leq_\mcs \bar{\gamma}$ iff the following all hold.\footnote{For example, 
suppose $\bar{\beta} = (0,0,5,\omega+ \omega, 2, 4)$ and $\bar{\gamma} = (\omega + 5, 3, 0, \omega + \omega, 15, 5)$. 
Then $\rn(\bar{\beta})$ is the set $\{ 0, 2, 4, 5, \omega + \omega \}$ and $\rn(\bar{\gamma})$ is the set $\{ 0, 3, 5, 15, \omega+5, \omega + \omega \}$. 
In this case $v = \{ \omega + \omega \}$. So we are left comparing $5$ and $\omega + 5$ and we conclude $\bar{\beta} \leq_\mcs \bar{\gamma}$.  
If $\rn(\bar{\beta}) \subsetneq \rn(\bar{\gamma})$, it could happen that every ordinal which appears only in $\bar{\gamma}$ is less than all of the 
ordinals appearing in $\bar{\beta}$, thus $v = \rn(\bar{\beta}) \subsetneq \rn(\bar{\gamma})$ and we are under the jurisdiction of the second case. 
(If one of the ordinals appearing only in $\bar{\gamma}$ is \emph{not} below all of the ordinals of $\bar{\beta}$, then $v \subsetneq \rn(\bar{\beta})$ so we are safely in the first case.)  
If we compare two sequences in which exactly the same sets of ordinals occur, just possibly in a different order or with different multiplicities, then $u$ is everything, so for definiteness use the lexicographic order.}

\begin{itemize}
\item $\bar{\beta}, \bar{\gamma} \in \mcs$.
\item let $v(\bar{\beta}, \bar{\gamma})$ be the maximal $v$ such that: $v \subseteq \rn(\bar{\beta})$ is an end segment and 
$v \subseteq \rn(\bar{\gamma})$ is an end segment in the sense of the usual ordering on ordinals. 
\item $\max (\rn(\bar{\beta}) \setminus v ) < \max (\rn(\bar{\gamma}) \setminus v)$ \underline{or} 
\\ $\rn(\bar{\beta}) \setminus v = \emptyset$ and $\rn(\bar{\gamma}) \setminus v \neq \emptyset$ \underline{or} 
\\ $\rn(\bar{\beta}) \setminus v = \rn(\bar{\gamma}) \setminus v = \emptyset$ and 
$\bar{\beta} \leq_{\operatorname{lex}} \bar{\gamma}$. 
\end{itemize} 
\end{quotation}
(A brief explanation in English is in the footnote.) This is a linear order on $\mcs_*$, and in fact a well-order.  

\br
For $\bar{\beta} \in \mcs$ let $\operatorname{Statement}({\bar{\beta}})$ be the following ``relativized'' statement of our chain condition for our  
system $\bar{N}$ of models. Note that the relativization arranges that in model $N_u$, we choose the element $\ma_u$ from $\ba_{\beta_u} \cap N_u$.  
\begin{quotation}
\noindent ${\operatorname{Statement}({\bar{\beta}}})$:  
\\ \noindent if $\ma_u \in N_{u} \cap \ba_{\beta_u}$ for $u \in  [\bn]^{\leq \bk}$ and 
$0 < \bigcap_u \ma_u$ 
then for every $n<\omega$, $\trv: n \rightarrow \{ 0, 1\}$ and distinct 
$\epsilon_0, \dots, \epsilon_{n-1} \in \del \setminus \bigcup \{ N_{ u } : u \in  [n]^{\leq k} \}$  
we have $\ma \cap \bigcap_{i<n} (\mx_{\epsilon_i})^{\trv(i)} > 0$.
\end{quotation}
This statement is intended to be evaluated in $\bM$.  Note that our desired conclusion of the proof is equal to 
$\operatorname{Statement}({\bar{\beta}})$ in the case where $\bar{\beta}$ is the sequence constantly equal to $\alpha_*$, 
i.e. $\bigwedge_u \beta_u = \alpha_*$. 
So it is enough to prove by induction on 
$\bar{\beta} \in \mcs$, in the order given by $\leq_\mcs$, that $\operatorname{Statement}({\bar{\beta}})$ holds.

The induction splits to cases. In each case, we will asssume the choice of $\ma_u$'s satisfy 
$\ma = \bigcap_u \ma_u > 0$, as otherwise there is nothing to prove.  We will use $\ba$ to denote $\ba_{\alpha_*}$, i.e. in $\bM$.  
We will often also write $\bigcap_u$ or $\bigwedge_u$ without writing the quantification $u \in [\bn]^{\leq \bk}$.

\bbr
\noindent \underline{Case 1}.  $\bigwedge_u \beta_u = 0$.  

This case is immediate from \ref{g20} and the choice of $\ba_0$.  

\bbr
\noindent \underline{Case 2}. For at least one $u$, call it $u_*$, $\beta_u$ is a limit ordinal.  (If there is more than one, fix the rightmost 
one in the sequence.)

Working in $\bM$, the set $\{ \mc \in \ba^+_{\beta_{u_*}} : \mc \in \bigcup \{ \ba_\alpha : \alpha < \beta_{u_*} \}$ 
and $\mc \leq \ma_{u_*}$ or $\mc \cap \ma_{u_*} = 0_\ba \}$ is dense in $\ba_{\beta_{u_*}}$. Hence there is a maximal antichain of 
$\ba_{\beta_{u_*}}$ comprised of such elements, and necessarily of size no more than $\mu$, by \ref{c:42}(1). Call this antichain $\mci$. Just 
as in the proof of \ref{g20}, since $\ma_{u_*} \in N_{u_*}$, without loss of generality $\mci \subseteq N_{u_*}$. 

By definition of construction sequence, $\ba_{\beta_{u_*}} \lessdot \ba$ so it is a regular subalgebra and its antichains remain antichains of $\ba$. 
Now in $\ba$, we have assumed that $\ma = \bigcap_u \ma_u > 0$, hence for some $\mc \in \mci$ we have that $\ma \cap \mc > 0$. Note that 
$\mc \cap \ma_{u_*} = 0$ is impossible, since $u_*$ is one of our $u$'s, so $\ma_{u_*} \geq \ma$. So as we chose $\mc$ from $\mci$, necessarily 
$\mc \leq \ma_{u_*}$. We made this choice in the big model, but because $\mci \subseteq N_{u_*}$, necessarily $\mc \in N_{u_*}$ and so 
$\mc \in N_{u_*} \cap \ba_{\beta_{u_*}}$. 

Define $\ma^\prime_u$ to be $\mc$ if $u = u_*$ and $\ma_u$ if $u \neq u_*$ (we change just one).  Correspondingly define 
$\beta^\prime_{u} = \min \{ \beta : \mb^\prime_u \in \ba_{\beta} \cap N_{u} \}$. Clearly this is well defined, and $\beta^\prime_u \leq \beta_u$ 
for every $u \neq u_*$, while $\beta^\prime_u < \beta_u$ for $u = u_*$ by our choice of $\mci$.  We have reduced to a strictly smaller tuple 
in the sense of $\leq_\mcs$, so by inductive hypothesis, we finish this case. 

\bbr
\noindent \underline{Case 3}.  Suppose that at least one of the $\beta_u$'s is not zero, and none are limit ordinals. 
Let $\beta_* = \max \{ \beta_u : u \in [\bn]^{\leq \bk} \}$, so $\beta_*$ is necessarily a successor, say $\beta_* = \gamma_* + 1$. 
Notice that there may be more than one $u$ for which $\beta_u  = \beta_*$.  
In due course, we will split into cases 3A and 3B according to whether the problem handled at stage $\gamma_*$ was from a $T_\xm$ or $\trg$ 
(as we shall explain). 

Working in $\bM$, by the definition of statement $(\bar{\beta})$ each $\ma_u$ also belongs to $\ba_{\beta_*}$, since the chain of Boolean algebras is increasing. 
So in $\bM$, $\ba_{\beta_*} \models \ma = \bigcap_u \ma_u > 0$. Moreover, if $\beta_u \neq \beta_*$ then in $\bM$, $\ma_u \in \ba_{\gamma_*}$. 

Let $\bar{\mb}$ be the possibility pattern that was solved at stage $\gamma_*$. 

Let $\langle u_\ell : \ell < \ell_* = \bn^\bk \rangle$ be a fixed enumeration of $\{ u :  u \in [\bn]^{\bk} \}$ (it could be the one from the beginning of the proof, but 
this isn't important). 

First we shall set $\ma_{-1} = \ma$. Then, by induction on $\ell < \ell_*$, we shall choose $\ma^\prime_{u_\ell}$ and $\ma_\ell$ to satisfy:
(i) $\ma^\prime_{u_\ell} \in N_{u_\ell} \cap \ba_{\beta_{u_\ell}}$, (ii) $0 < \ma^\prime_{u_\ell} \leq \ma_{u_\ell}$, and (iii) 
$\ma_\ell := \bigcap_{k \leq \ell} \ma^\prime_{u_k} \cap \bigcap_{k > \ell} \ma_{u_k} > 0$. 
We do so  
as follows:\footnote{As usual, 
we would like to replace $\ma_{u_\ell}$ with (some smaller nonzero combination of elements from an earlier Boolean algebra and 
generators of the formal solution) as in \ref{smooth}(1); but here we want to make sure, in addition, that the replacement is intelligible to $N_{u_\ell}$ and 
also that the common intersection $\ma$ remains nonzero in $\ba$. The simple acrobatics described carry this out.}

\begin{itemize}
\item initially, as said, $\ma_{-1} = \ma$. 
\item if $\beta_{u_\ell} \neq \beta_*$, then $\ma^\prime_{u_\ell} = \ma_{u_\ell}$ and $\ma_\ell = \ma_{\ell-1}$. 
\item if $\beta_{u_\ell} = \beta_*$, proceed as follows. Working in $\bM$, call an element of $\ba^+_{\beta_*}$  \emph{special} if it can be written as 
$\mx \cap \mb^1_w \cap \bigcap_{n < t} (-\mb^1_{v_n})$, where $\mx \in \ba^+_{\gamma_*}$, $t \in \mathbb{N}$ 
and $w$ and each $v_n$ $(n<t)$ is a finite subset of 
$\lambda$, with each $v_n \not \subseteq w$. The special elements are dense in the completion, so dense in $\ba_{\beta_*}$. So we may choose 
$\mci$ a maximal antichain of $\ba_{\beta_*}$ consisting of special elements $\mz$ such that either $\mz \leq \ma_{u_\ell}$ or $\mz \cap \ma_{u_\ell} = 0$.  
Since $\ma_{u_\ell} \in N_{u_\ell}$, $\ba_{\beta_*} \in N_{u_\ell}$, without loss of generality $\mci \in N_{u_\ell}$ hence 
 $\mci \subseteq N_{u_\ell}$. Since $\ba_{\beta_*}$ is a regular subalgebra of 
$\ba$ and its maximal antichains remain maximal antichains, there is some $\mz \in \mci$ such that $\ba \models \ma_{\ell-1} \cap \mz > 0$. 
By our choice of $\mci$ (and recalling that $\ma \leq \ma_{u_\ell}$), 
necessarily $\ma_{\ell-1} \cap \mz \leq \ma_{u_\ell}$.  To complete this stage, define $\ma^\prime_{u_\ell} = \mz$ and 
define $\ma_\ell = \ma_{\ell-1} \cap \mz$.  Notice it is still the case that $\ma^\prime_{u_\ell} \in N_{u_\ell} \cap \ba_{\beta_{u_\ell}}$ and that 
the common intersection is nonzero.  
\end{itemize}

So we have reduced to the following case: if $\beta_{u_\ell} = \beta_*$, we may assume $\ma_{u_\ell} = \mx_\ell \cap \mb^1_{w_\ell} \cap \bigcap_{n<t_\ell} (-\mb^1_{v^\ell_n})$. 
If $\beta_{u_\ell} \neq \beta_*$, then $\ma_{u_\ell} = \mx_\ell$. Now as $\bigcap_\ell \ma_{u_\ell} > 0$, necessarily $w_\ell \not\subseteq v^i_n$ for $n < t_\ell$, 
$\ell < \ell_*$. Also the ``potential'' $\epsilon$'s can be from $\bigcup \{ v^\ell_n : n < t_\ell, \ell < \ell_* \}$. So we can ignore the ``$-\mb^1_{v^\ell_n}$'' terms. 

That is:  if $\beta_{u_\ell} = \beta_*$, assume $\ma_{u_\ell} = \mx_\ell \cap \mb^1_{w_\ell}$, and  
if $\beta_{u_\ell} \neq \beta_*$, then $\ma_{u_\ell} = \mx_\ell$.

We'll need one final reduction: Suppose at this stage we're given a fixed finite set of quantifier-free formulas 
$\Sigma$ of $T$, the theory associated to the possibility pattern $\langle \mb_u : u \subseteq \bigcup_{\ell < \ell_*} w_\ell \rangle$.  ($\Sigma$ must be fixed at this point, but 
may depend on information obtained after the initial application of \ref{smooth}, such as the size of the $w$'s.) 
We may be tempted to assume that each $\mx_\ell$ decides all formulas from $\Sigma$, but actually it can decide only $\Sigma_\ell$, a finite subset of the possibility 
pattern for $\langle \mb_w : w \subseteq N_{u_\ell} \cap \gamma_* $ finite $\rangle$ which necessarily belongs to $N_{u_\ell}$, and will be chosen later. 
Define $W$ as the finite set of relevant $\ma$'s (=formulas, in the notation of \ref{ntn:a}). 
By induction on $\ell < \ell_*$, define $\mx^\prime_{\ell}$ and $\ma_\ell$ 
so that (i) $\mx^\prime_\ell \in \ba_\beta \cap N_{u_\ell}$, 
(ii) $0 < \mx^\prime_\ell \leq \mx_\ell$, (iii) $\mx^\prime_\ell$ is decisive in the sense defined below,  
(iv) $\ma_\ell = \ma_{\ell-1} \cap \mx^\prime_\ell \cap \mb^1_{w_\ell} > 0$. 
We proceed as follows:
\begin{itemize}
\item initially, $\ma_{-1} = \ma$ [recall this was the intersection $\bigcap_u \ma_u > 0$].
\item for each $\ell < \ell_*$, proceed as follows. Let $\beta = \beta_{u_\ell}$ if $\beta_{u_\ell} \neq \beta_*$ and let $\beta = \gamma_*$ otherwise. 
(So either way, $\beta \leq \gamma_*$ and $\mx_\ell \in \ba_\beta \cap N_{u_\ell}$.)
Working in $\bM$, call an element of $\ba^+_{\beta}$  \emph{decisive} if it decides all formulas in $\Sigma_\ell$. Choose 
$\mci$ a maximal antichain consisting of decisive elements $\mz$ such that either $\mz \leq \mx_{\ell}$ or $\mz \cap \mx_{\ell} = 0$.  
Since $\mx_{\ell} \in \ba_{\beta} \cap N_{u_{\ell}}$, without loss of generality $\mci \subseteq N_{u_\ell}$. 
Since $\ba_{\beta}$ is a regular subalgebra of 
$\ba$, there is some $\mz \in \mci$ such that $\ba \models \ma_{\ell-1} \cap \ma_{u_\ell} \cap \mz > 0$, so 
necessarily $\ma_{\ell-1} \cap \mz \cap \ma_{u_\ell} \leq \mx_{\ell}$.  To complete this stage, define $\mx^\prime_{\ell} = \mz \cap \mx$ and 
define $\ma_\ell = \ma_{\ell-1} \cap \mz \cap \mb^1_{w_\ell}$. 
\end{itemize}
To summarize, we have reduced to the case where each $\ma_{u_\ell}$ is of the form $\mx_\ell \cap \mb^1_{w_\ell}$ where 
each $\mx_\ell \in N_{u_\ell} \cap \ba_{\gamma_*}$, each $w_\ell$ is finite, each $\mx_\ell$ decides all formulas from $\Sigma_\ell$, and 
finally, $\ba \models \ma = \bigcap_u \ma_u > 0$. 

To finish proving $\operatorname{Statement}(\bar{\beta})$, it would suffice to show that 
\begin{equation} \label{almost-the-end}
\bigcap \{ \mx_\ell : \beta_{u_\ell} \neq \beta_* \} \cap \bigcap \{ \mx_\ell : \beta_{u_\ell} = \beta_* \} \cap \mb_W > 0 
\end{equation} 
where $W = \bigcup \{ w_\ell : \beta_{u_\ell} = \beta_* \}$. Why? Then, working in $\bM$, there is a homomorphism of 
$\ba_{\beta_*}$ onto $\ba_{\gamma_*}$ which is the identity on $\ba_{\gamma_*}$ and takes each $\mb^1_\gamma$ (for $\gamma \in W$) to $\mb_W$, and 
$\mb^1_\gamma$ (for $\gamma \in \lambda \setminus W$) to $0$.  
This allows us to replace each $\ma_{u_\ell}$ by $\mx_\ell \cap \mb_W$, which is an element of the earlier Boolean algebra $\ba_{\gamma_*}$, 
and thus to apply the inductive hypothesis to a smaller $\bar{\beta}$. (More plainly: because of the homorphism, nonzero intersections with the 
distinguished sequence in the earlier algebras must reflect nonzero intersections in the present, so we would finish the proof.) 

For (\ref{almost-the-end}), there are two cases depending on the identity of the theory.  Remember from \S \ref{s:locality} that dealing with the formulas mentioned 
in each case will suffice.

\bbr
\noindent \underline{Case 3A: $T = \trg$}.  In this case, we remember that our possibility pattern $\bar{\mb}$ 
reflected a type $\{ R(x,a_{(\gamma, 1)}) \land \neg R(x, a_{(\gamma, 0)}) : \gamma < \lambda \}$, since dealing with this formula suffices by \S \ref{s:locality}. 
We take $\Sigma_\ell$ corresponding to deciding $\ma[a_{(\zeta, \ii)} = a_{(\delta, \ii)}]$ for 
$\zeta, \delta \in W$, $\ii \in \{ 0, 1 \}$. 
Because the common intersection $\ma > 0$, it is necessarily the case that if $\zeta \in w_\ell$ and $\delta \in w_{\ell^\prime}$ then 
$\ma \cap \ma[a_{(\zeta, \ii)} = a_{(\delta, 1-\ii)}] = 0$. So $\ma \leq \mb_W$ as desired.\footnote{In the earlier proof, 
we did not assume a priori that the $\ma$'s had a common intersection, so we had to smooth things out using the $\Delta$-system lemma 
and the like to obtain one, below $\mb_W$.  In the present proof, we don't have the luxury of throwing away some of the $\ma$'s, but conversely we know in advance 
that their intersection is nonzero, which compensates.}

\bbr
\noindent \underline{Case 3B: $T = T_\xm$}.
In this case, we remember that our possibility pattern $\bar{\mb}$ 
reflected a type of the form $\{Q_\nu(x) \land R(x,\bar{a}_{v_\gamma}) : \gamma < \lambda \}$, again by \S \ref{s:locality}. 

By quantifier elimination, there is some finite $M$ depending on $|W|$ so that whether $\bigwedge \{ R(x,\bar{y}_{v_\gamma}) : \gamma \in W \}$ 
is consistent depends only on the types of $\{ y_\zeta : \zeta \in v_\gamma, \gamma \in W \}$ up to level $M$ in $\mct_2$. 
We take $\Sigma_\ell$ corresponding to deciding $\ma[P_\eta(a_\zeta)]$ for  $\eta \in \mct_{2}, \ell(\eta) \leq M, \zeta \in v_\gamma, \gamma \in W$, 
and for good measure, also to deciding $\ma[a_\zeta = a_\delta]$ for $\zeta, \delta \in v_\gamma$ and $\gamma \in W$. 
Here too, since the common intersection $\ma > 0$, the answer to consistency (which by choice of $\Sigma$ is either uniformly yes or uniformly no) 
is uniformly yes: that is, we have that $\ma \leq \mb_W$ which finishes the proof. 
\end{proof}

\br
\begin{concl}  \label{c:constr} 
Suppose $\lambda = \del \geq \mu^+ > \aleph_0$. Then:
\begin{enumerate}
\item[(a)]   Any construction sequence produces a Boolean algebra $\ba_* = \ba_{\alpha_*}$ of size $\leq (|\ba_0| + \lambda)^\lambda$, along with an 
ultrafilter $\de_*$ on $\ba_*$.  
\item[(b)] $(\ba_*, \bar{\mx})$ satisfies the 
$(\bn, \bk, \mu)$-c.c.  for every finite $\bn > \bk \geq 2$ for $\bar{\mx} = \langle \mx_{(\epsilon,0)} : \epsilon < \del \rangle$, 
and also satisfies the $(\mu^+, m)$-Knaster condition for any finite $m$, thus a fortiori the $\mu^+$-c.c.   
\item[(c)] If in addition $\alpha_* = 2^\lambda$, then for some construction sequence, 
$\de_*$ will be moral for the theory of the random graph and for every $T_\xm$, in the sense of being able to handle any possibility pattern 
$\langle \mb_u : u \in [\lambda]^{<\aleph_0} \rangle$.   
\end{enumerate} 
\end{concl}

\begin{proof}
Item (a) on size, follows by induction as each $\ba_\alpha$ has the $\mu^+$-c.c.  Item (b), on chain conditions and 
the Knaster property, follow from \ref{c:42} and \ref{cc-lemma}.

Regarding item (c): if $\alpha_* = 2^\lambda$, then we have enough room to enumerate and handle all the relevant possibility patterns. 
We don't know the final Boolean algebra in advance, of course, so we can't enumerate in advance ``possible possibility patterns'' involving its subsets. 
We can handle this either by interpolating into our enumeration at each stage everything that becomes a possibility pattern by then,  
or alternately as follows. 

Recalling separation of variables, say that $\langle B_u : u \in [\lambda]^{<\aleph_0} \rangle$ is a \emph{pre-possibility pattern} if it is a set of nonempty subsets of $\lambda$ which is monotonic in the sense that $u \subseteq v$ implies $B_u \supseteq B_v$.  Say that $(q, p)$
is a \emph{pre-type} for $(T, \vp(x,\bar{y}))$ if $q$ is a partial type in the infinitely many variables $x, \langle y_\gamma : \gamma < \lambda \rangle$ 
which implies the complete $T$-type of each parameter variable $y_\gamma$, and asserts that $p = \{ \vp(x,\bar{y}_{v_\alpha}) : \alpha < \lambda \}$
is a partial type for $\langle v_\alpha : \alpha < \lambda \rangle$ some enumeration of some set of subsets of $\lambda$ with each $v_\alpha$ of size $\ell(\bar{y})$. 
 
Say that \emph{$\bar{B}$ represents $(q, p)$} when the following holds. 
There exist a regular ultrafilter $\de$ on $\lambda$ and 
 $\jj: \mcp(\lambda) \rightarrow \ba_\alpha$ satisfying the hypotheses of separation of variables for $(\ba_\alpha, \de_\alpha)$, and 
 an enveloping ultrapower $N = M^\lambda/\de$, and $\langle a_\gamma : \gamma < \lambda \rangle$ a sequence of elements of $M^I$, 
 such that $q_\beta(x,\langle a_\gamma/\de : \gamma < \lambda \rangle)$ is a partial type, thus  $p_\beta(x) = \{ \vp(x,\bar{a}_{v_\alpha}/\de) : \alpha < \lambda \}$ is  a partial type, and finally, for each $u \in [\lambda]^{<\aleph_0}$, we have that $\{ t \in \lambda : M \models \exists x \bigwedge \{ \vp(x,\bar{a}_{v_\alpha}(t)) : \alpha \in u \} = B_u$.
 
 Let $\langle (\bar{B}_\beta, (q_\beta, p_\beta)) : \beta < 2^\lambda \rangle$ 
 be an enumeration of all pairs of pre-possibility patterns and associated pre-types for each of our $(T, \vp)$'s of interest, 
 each occurring cofinally often. 
 At stage $\alpha = \beta+1$, look at $(\ba_\alpha, \de_\alpha)$ and ask the following. 
 Does $\bar{B}_\beta$ represent $(q_\beta, p_\beta)$? If so, is $\jj(\bar{B}_\beta) \subseteq \de_\alpha$? 
 If the answer to both is yes, then at this stage we solve the possibility pattern $\jj(\bar{B}_\beta)$. If not, do nothing (or 
 if some action is desired, re-solve a previously solved possibility pattern).   In this way any relevant problem that may show up in the 
 final $(\ba_*, \de_*)$ is eventually handled, recalling that the 
 cofinality of the construction is strictly greater than $\lambda$. 
\end{proof}

\newpage

\section{First direction: non-saturation}  \label{s:n-s}

In this section we show that preserving the chain condition from \ref{g2a} will block saturation of $T_{n,k}$. 
We 
know from Conclusion \ref{c:constr} in the  
previous section that we can preserve this chain condition while adding formal solutions to theories $T_\xm$. Together these give a proof of Theorem 
\ref{t:first} below, explaining how to build regular ultrafilters which are good for the $T_\xm$'s and not good for the $T_{n,k}$'s. 

Recall that for an integer $k$ and infinite cardinals $\lambda \geq \mu$, 
we call $F: [\lambda]^k \longrightarrow [\lambda]^{<\mu}$ a \emph{set mapping} if $F(\sigma) \cap \sigma  = \emptyset$ for all $\sigma \in [\lambda]^k$. 
Write $(\lambda, k, \mu) \longrightarrow n$ to mean that for every set mapping $F: [\lambda]^k \longrightarrow [\lambda]^{<\mu}$ there 
is a free set of size $n$, i.e. there is $w \in [\lambda]^n$ such that $F(\sigma) \cap w = \emptyset$ for all $\sigma \in [w]^k$. 
Half of a well known characterization of Kuratowski-Sierpinski is the theorem that 
\[ (\aleph_{\alpha + k}, k, \aleph_\alpha) \longrightarrow k+1 \]
for any ordinal $\alpha$, see \cite[\S 46]{ehmr}. 
It follows by monotonicty that for any integers $n > k$, 
\[ (\aleph_{\alpha + n}, k, \aleph_\alpha) \longrightarrow n+1 \]
see for instance \cite{MiSh:1050} Corollary 1.3. 
For our purposes here and in \cite{MiSh:1050}, it was convenient to replace set mappings with \emph{strong} set mappings 
(which replace the requirement that $F(\sigma) \cap \sigma = \emptyset$ with the requirement that $\sigma \subseteq F(\sigma)$, for all 
$\sigma \in [\lambda]^k$). This allows us to think of $F$ as setting some kind of closure for $\sigma$, and does not bother the 
freeness result, as the next observation shows. 

\begin{defn} \label{d56} 
Write $(\lambda, k, \mu) \xrightarrow{strong} n$ to mean that for every \emph{strong} set mapping $G: [\lambda]^k \longrightarrow [\lambda]^{<\mu}$ there is $w \in [\lambda]^n$ such that $w \not\subseteq G(u)$ for all $u \in [w]^k$.   
\end{defn}

\begin{obs} \label{o:ssm}
\begin{equation}
\label{e0}
(\lambda, k, \mu) \longrightarrow n \mbox{ \hspace{3mm} iff \hspace{3mm} } (\lambda, k, \mu) \xrightarrow{strong} n
\end{equation}
in particular
\begin{equation}
\label{e1} (\aleph_{\alpha + k}, k, \aleph_\alpha) \xrightarrow{strong} k+1 
\end{equation}
and by monotonicity, if $k< n$, 
\begin{equation} (\aleph_{\alpha + n}, k, \aleph_\alpha) \xrightarrow{strong} n+1. 
\end{equation}
\end{obs}

\begin{proof}
For (a), (b), (c) it suffices to prove (a).  $(\leftarrow)$ Let $G: [\lambda]^k \rightarrow [\lambda]^{<\mu}$ be a strong set mapping. 
Define $F: [\lambda]^k \rightarrow [\lambda]^{<\mu}$ by $F(\sigma) = G(\sigma) \setminus \sigma$ for $\sigma \in [\lambda]^k$. 
Let $w$ be such that for any $\sigma \in [w]^k$, $F(\sigma) \cap w = \emptyset$, so recalling 
$G(\sigma) = F(\sigma) \cup \sigma$, we have $G(\sigma) \cap w \subseteq \sigma \subsetneq w$. ($\rightarrow$) Similarly. 
\end{proof}

For our present purposes, to deal with $n$ possibly much larger than $k$, the following is what we need.  Note that 
in \ref{d56a}, each $\tau$ escapes `capture' not only by its own subsets of size $k$, but indeed by any subset of $w$ of size $k$. 

\begin{defn} \label{d56a} 
Write $(\lambda, k, \mu) \xrightarrow{stronger} n$ to mean that for every {strong} set mapping $G: [\lambda]^k \longrightarrow [\lambda]^{<\mu}$ there is $w \in [\lambda]^n$ such that 
$\forall \sigma \in [w]^k$  $\forall \tau \in [w]^{k+1}$ we have $(\tau \not\subseteq F(\sigma))$. 
\end{defn}

\begin{rmk}
It is a theorem of ZFC that $(\mu^{+\omega}, k, \mu) \xrightarrow{stronger} n$, and even \\ $(\mu^{+n}, k, \mu) \xrightarrow{stronger} n$. 
However, if we fix $n,k$ and we would like to get $(\mu^{+k+7}, T_{n,k})$-non-morality or so, we run into 
consistency problems $($see \cite{KoSh:645} and \cite{MoSh:1072}$)$. 
\end{rmk}

The main work of this section is in proving Lemma \ref{f2c}, which can be seen as a strong replacement for:

\begin{fact}[\cite{MiSh:1050} Claim 5.1]
Suppose $\ba = \ba^1_{2^\lambda, \mu}$ and $(\lambda, k, \mu^+) \xrightarrow{strong} k+1$.  Then no ultrafilter 
on $\ba$ can be $\lambda^+$-moral for $T_{k+1,k}$.
\end{fact}

\noindent not only in the sense of allowing $n$ to have large finite distance from $k$, but also in the sense 
that instead of using the completion of a free Boolean algebra $\ba^1_{2^\lambda, \mu, \aleph_0}$, 
we may use any Boolean algebra with a little freeness, as described by the c.c. 

\br 

\begin{lemma} \label{f2c}
Suppose that:  
\begin{enumerate}
\item 
$(\lambda, k, \mu^+) \xrightarrow{stronger} n+1$
\item $\ba$ is a complete Boolean algebra such that: 
\begin{enumerate} 
\item $\ba^1_{\lambda, \mu, \aleph_0} \lessdot \ba$ and 
\item $(\ba, \bar{\mx})$ satisfies the $(n, k, \mu)$-c.c.  
\\ where $\bar{\mx} = \langle \mx_{\epsilon} : \epsilon < \lambda \rangle$ and\footnote{There is 
such a free set of size $\lambda$ none of whose elements belong to $\de$.} 
each 
$\mx_\epsilon \in \{ \mx_{(\epsilon, 0)} \cap - \mx_{(\epsilon, 1)}, \mx_{(\epsilon, 1)} \} \setminus \de$. 
\end{enumerate}
\item $\de$ is an ultrafilter on $\ba$ 
\item $T = T_{n,k}$  
\end{enumerate}
Then $\de$ is not $(\lambda, T)$-moral,\footnote{In \cite{MiSh:999} we defined moral without the plus: 
$\lambda$-moral means over sets of size $\lambda$.}
 i.e. there is a possibility pattern $\langle \mb_u : u \in [\lambda]^{<\aleph_0} \rangle$ 
with no multiplicative refinement. 
\end{lemma}

\begin{rmk} \label{rmk42}
The cardinal conditions of $\ref{f2c}(1)$ and $\ref{conv-a}$ are satisfied in ZFC when $\del = \lambda = \aleph_\omega$, $\mu = \theta = \aleph_0$.
\end{rmk}

\begin{proof}[Proof of Lemma \ref{f2c}]  Recall that in $T = T_{n,k}$, $R$ is a $(k+1)$-place symmetric irreflexive relation and the forbidden configuration has size $(n+1)$. 
Fix $M \models T$ and $N = M^I/\ee$ an enveloping ultrapower (\ref{d:enveloping} above). In $N$, or in the monster model for $T$, let 
\[ B = \{ b_{\alpha, i} : \alpha < \lambda, i < n \} \]
be a rectangle of distinct parameters [i.e. $b_{\alpha,i} = b_{\alpha^\prime, i^\prime}$ if and only if $(\alpha, i) = (\alpha^\prime, i^\prime)$] on which there are 
no instances of $R$.  
The intent will be that our type $p(x)$ asserts that $x$ connects to every $k$-tuple of $b$'s which is doubly strictly increasing, i.e. 
strictly increasing in each coordinate. Since there are no background edges on $B$ (at least modulo the ultrafilter) this is a type.  However, we will choose 
our possibility pattern so that some edges ``briefly and occasionally'' appear. 

More precisely, define $\mathbb{P}$ (the set of indices for parameters) by 
\begin{align*} 
\mathbb{P} = \{ (\bar{\beta}, \bar{i}) ~:~ &  (\bar{\beta}, \bar{i}) \mbox{ is doubly strictly increasing, i.e.  } \\
& \bar{\beta} \mbox{ is a strictly increasing sequence of elements of $\lambda$ of length $k$, } \\
& \bar{i} \mbox{ is a strictly increasing sequence of elements of $n$ of length $k$.} \} \\
\end{align*}
[This notation just separates the first and second coordinates of doubly strictly increasing sequences of length $k$.] 
Fix also an enumeration of 
$\mathbb{P}$ as 
\[ \langle (\bar{\beta}, \bar{i})_\alpha : \alpha < \lambda \rangle. \] 
We can recover the corresponding set of elements of $B$ by writing 
\[ \bar{b}_{(\bar{\beta}, \bar{i})} \mbox{ as shorthand for }\langle b_{(\beta_0, i_0)}, \dots, b_{(\beta_{k-1}, i_{k-1})} \rangle \] 
where $\bar{\beta} = \langle \beta_0, \dots, \beta_{k-1} \rangle$ 
and $\bar{i} = \langle i_0, \dots, i_{k-1} \rangle$. 
Our type will be 
\[ p(x) = \{ \vp_\alpha := R(x,\bar{b}_{ (\bar{\beta}, \bar{i})_\alpha }): \alpha < \lambda \} \]
again, asserting $x$ connects to every doubly strictly increasing sequence of elements in our rectangle $B$. 
Next we define $\mathbb{E}$ (the set of indices for shadow edges) by 
\begin{align*} 
 \mathbb{E} = \{ (\bar{\gamma}, \bar{j}) ~:&~  \bar{\gamma} \mbox{  a strictly increasing sequence of elements of $\lambda$ of length $k + 1$, }  \\
& \bar{j} \mbox{ a strictly increasing sequence of elements of $n$ of length $k + 1$.} \} 
\end{align*} 
[Note $k+1$ instead of $k$ this time.]  
Fix a bijection $f: \mathbb{E} \rightarrow \{ \mx_{\epsilon} : \epsilon < \lambda \}$, which gives an enumeration 
\[ \langle  (\bar{\gamma}, \bar{j})_\epsilon : \epsilon < \lambda \rangle. \]
and a way of associating to each $(\bar{\gamma}, \bar{j})_\epsilon$ the element $\mx_{\epsilon}$ of our free sequence.  

Now we work towards a possibility pattern. First, for any 
$(\bar{\gamma}, \bar{j})_\epsilon \in \mathbb{E}$, 
define  
\[ \ma[R(\bar{b}_{(\bar{\gamma}, \bar{j})_\epsilon})] ~ = ~ \mx_{\epsilon}. \]
If $(\bar{\gamma}, \bar{j})$ has length $k+1$ but is not doubly strictly increasing, set $\ma[R(\bar{b}_{(\bar{\gamma}, \bar{j})})]  = 0$. 
Set $\ma[\bar{b}_{(\bar{\gamma}_1, \bar{j}_1)} = \bar{b}_{(\bar{\gamma}_2, \bar{j}_2)}]$ to be $0$ if ${(\bar{\gamma}_1, \bar{j}_1)} \neq 
{(\bar{\gamma}_2, \bar{j}_2)}$ and $1$ if ${(\bar{\gamma}_1, \bar{j}_1)} =  
{(\bar{\gamma}_2, \bar{j}_2)}$. (Since $E$ contains only doubly strictly increasing sequences, let us spell out fully the rest of the conditions 
given by the theory. 
The edge relation is irreflexive, meaning $R(x_0, \dots, x_{k}) \implies \bigwedge_{i<j\leq k} x_i \neq x_j$,  and symmetric, meaning 
$R(x_0, \dots, x_{k}) \implies R(x_{\pi(0)}, \dots, x_{\pi(k)})$ for any bijection $\pi: k+1 \rightarrow k+1$. Finally, for every 
$i_0 < \cdots < i_k < n+1$, we have that $R(x_{i_0}, \dots, x_{i_k})$ is forbidden.) 

All cases of $\ma[R(b_{(\epsilon_0, i_0)}, \dots, b_{(\epsilon_{i_{k}}, i_{k})}]$ not just stated or not implied by the cases just stated are set to be 0.
Finally,  for each $u \in [\lambda]^{<\aleph_0}$, define 
$\mb_u = 1 - \bigcup \{ \my_{\bar{\beta}} : \bar{\beta} $ an increasing sequence of length $n$ from $u$, where 
$\my_{\bar{\beta}} = \bigcap \{ \mx_\epsilon : $ for some $i_0 < \cdots < i_k < n$ we have $(\langle \beta_{i_\ell} : \ell \leq k \rangle, \langle i_\ell : \ell \leq k \rangle ) 
= (\bar{\beta}, \bar{i})_\epsilon \}~ \}$. 

\br
Let 
\[ \bar{\mb} = \langle \mb_u : u \in [\lambda]^{<\aleph_0} \rangle. \]
This completes the definition of our possibility pattern.  Note that the fact that our rectangle is bounded by $n$ in one direction was crucial for this to be a possibility pattern: 
even if we add edges on \emph{all} doubly strictly increasing $(k+1)$-element sequences of elements of $B$ we will never get a clique on $(n+1)$-vertices. 

Assume for a contradiction that $\langle \mb^1_\alpha : \alpha < \lambda \rangle$ is a multiplicative refinement for $\bar{\mb}$.   
It will be useful to have the following translation between finite subsets $u$ of $\lambda$ in the sense of formulas in the type, and finite subsets $\sigma$ of $\lambda$ 
in the sense of first coordinates in our array $B$, since one has to do with our multiplicative refinement and the other with the domain of our partition theorem. 
So: 
\begin{itemize}
\item for any $u \in [\lambda]^{<\aleph_0}$, define 
\[ \proj(u) = \{ \beta < \lambda : b_{\beta,i} \mbox{ occurs in $\bar{b}_{(\bar{\beta}, \bar{i})_\alpha}$ for some 
$\alpha \in u \}$ }. \]
So $\proj(u) \in [\lambda]^{<\aleph_0}$, but $\lambda$ is now the set of 
first coordinates of elements of $B$, no longer the index set for the type. Note: if $u \neq \emptyset$, 
$|\proj(u)| \geq k$. 
\item for any $\sigma \in [\lambda]^{<\aleph_0}$, define 
\[ \cl(\sigma) = \{ \alpha < \lambda : \proj( \{ \alpha \} ) \subseteq \sigma \}. \] 
Visually, if $u = \{ \alpha_0, \dots, \alpha_\ell \}$ then 
$\proj(u)$ is the set of first coordinates of (indices for) parameters appearing in $\vp_{\alpha_0}, \dots, \vp_{\alpha_\ell}$.  
Continuing, $\cl(\proj(u))$ is again a set of indices for formulas in the type, 
namely, all formulas with parameters coming from doubly strictly increasing sequences with first coordinates from $\proj(u)$. 
\end{itemize}
Let $\bM = (\mch(\chi); \in)$ 
and let $\bM^+$ be an expansion by Skolem functions. 
Let $x$ be an element coding $(\ba, \bar{\mx}, \mu, \aleph_0, \lambda)$. 
Define $F: [\lambda]^{\leq k} \rightarrow [\lambda]^{\leq \mu}$ by: 
\[ \sigma \mapsto \lambda \cap \operatorname{Sk}\left( \sigma \cup \cl(\sigma) \cup \{ f, x \} \cup \{ \mb^1_\alpha : \alpha \in \cl(\sigma) \} \cup \mu \cup \{ \lambda, \mu, \aleph_0 \},  \bM^+\right). \]
So $F$ is well defined and its range is a subset of $\lambda$ of size $\mu$ (because we are taking the Skolem hull in a countable language of a set of size $\mu$; note that 
the set whose Skolem hull we take includes $\mu$ as a set). For the partition theorem, we focus on $F \rstr [\lambda]^k$, but define it on $[\lambda]^{\leq k}$ 
to match Definition \ref{d:n-k-pos}. 

By assumption (1) of the Lemma, there is some $w \in [\lambda]^n$ such that for all $\sigma \in [w]^k$ and $\tau \in [w]^{k+1}$ we have $\tau \not\subseteq F(\sigma)$ (or just that $F(\sigma) \cap w) = \sigma$). 
Fix this $w$ for the rest of the proof.

Defining $N_\sigma = F(\sigma)$ for $\sigma \in [w]^{\leq k}$ gives a family of submodels of $\bM$ (taking, if desired, the reduct to the language without Skolem functions). 
These models are in $(n,k, \mu)$-general position, since we have the necessarily monotonicity from $F$. 
For each $\sigma \subseteq w$, choose $\ma_\sigma = \bigcap \{ \mb^1_\alpha : \alpha \in \cl(\sigma) \}$, which is possible. 
Also $\ma_\sigma \in N_\sigma$ by definition. We know that 
\[ \bigcap \{ \ma_\sigma : \sigma \in [w]^k \} > 0 \]
since $\bar{\mb}^1$ is a multiplicative refinement (every finite intersection is in $\de$, hence nonzero). Let $(\bar{\delta}, \bar{\ell})$ be the unique doubly strictly increasing sequence 
of length $n$ whose first coordinates are from $w$.   Suppose that $(\bar{\gamma}, \bar{j}) = (\bar{\gamma}, \bar{j})_\epsilon$ 
is a strictly
 increasing subsequence of $(\bar{\delta}, \bar{\ell})$ of length $k+1$.  Then $\mx_\epsilon \notin N_\sigma$ for any $\sigma \in [w]^k$. 
[Why? Assume for a contradiction that $\mx_\epsilon$ were in some $N_\sigma$. Recall that $N_\sigma$ contains the bijection $f$, so 
from $\mx_\epsilon$ we can define $(\bar{\gamma}, \bar{j})_\epsilon$ from which we can define the projection on to the first coordinates 
$\bar{\gamma}$, so the set $\tau$ corresponding to the sequence $\bar{\gamma}$ is contained in $N_\sigma = F(\sigma)$. But $\tau \in [w]^{k+1}$ 
so we have found $\tau \in [w]^{k+1}$ and $\sigma \in [w]^k$ such that $\tau \subseteq F(\sigma)$, contradicting our partition theorem.] This was for 
$\sigma \in [w]^k$, but
by monotonicity, it follows that $\mx_\epsilon \notin N_\sigma$ for any $\sigma \in [w]^{\leq k}$. 

Let $E = \{ \epsilon ~: ~(\bar{\gamma}, \bar{j})_\epsilon$ 
is a doubly strictly increasing subsequence of $(\bar{\delta}, \bar{\ell})$ of length $k+1 \}$. We have just seen that
\[ E \subseteq \lambda \setminus \bigcup \{ N_\sigma : \sigma \in [n]^{\leq k} \}. \]
(It does not matter 
whether indices $\epsilon$ or elements $\mx_\epsilon$ are excluded, as the one-to-one mapping 
$\epsilon \mapsto \mx_\epsilon$ belongs to $\bigcap \{ N_\sigma : \sigma \in [w]^k \}$.)
So by our chain condition,
\[ (\star) ~~ \bigcap \{ \ma_\sigma : \sigma \in [w]^k \} \cap \{ \mx_\epsilon : \epsilon \in E \} > 0. \]
Recall, however, that $\ma_\sigma = \bigcap \{ \mb^1_\alpha : \alpha \in \cl(\sigma) \}$, so the intersection $(\star)$ includes 
$\bigcap \{ \mb^1_\alpha : (\bar{\beta}, \bar{i})_\alpha$ is a $k$-element subsequence of $(\bar{\delta}, \bar{\ell}) \}$, which cannot have 
positive intersection with the set $\bigcap \{ \mx_\epsilon : \epsilon \in E \}$ if $\bar{\mb}^1$ is a multiplicative refinement of $\bar{\mb}$. 
\end{proof}

\begin{theorem} \label{t:first}  Let $\del = \lambda$, $\mu$, $\theta = \aleph_0$ satisfy $\ref{f2c}(1)$ and $\ref{conv-a}$, for example, 
$\del = \lambda = \aleph_\omega$, $\mu = \aleph_0$.  Fix $n > k \geq 2$. 
Then there is a regular ultrafilter $\de$ on $\lambda$ which is:
\begin{enumerate}
\item good for the theory of the random graph, 
\item good for every $T_\xm$, 
\item not good for $T_{n,k}$. 
\end{enumerate}
\end{theorem}

\begin{proof}  
(1), (2) are given by Conclusion \ref{c:constr} and (3) is Lemma \ref{f2c}. 
\end{proof}

\newpage
\section{Second direction} \label{s:second}

In this section we show that it is possible to saturate $T_{n,k}$ 
while not saturating $T_\xm$. 
Our strategy is to use the key chain condition from the proof in \cite{MiSh:1167} that Keisler's order has the maximum number of classes, see \ref{d:1167-cc} below. 
The main work of the section is in Lemma \ref{l:add-tnk-solution} which shows that adding a formal solution to a positive 
problem from $T_{n,k}$ can easily be done while preserving this chain condition, with 
no special assumptions on the ideal $\mci$ mentioned in \ref{d:1167-cc}. Since the ultrafilters we consider will be good for the theory of the random graph, solving 
positive $T_{n,k}$-problems suffices.

\begin{defn}[The $(\kappa, \mci, \bar{m})$-c.c., \cite{MiSh:1167} Definition 8.2] \label{d:1167-cc}
Let $\kappa$ be an uncountable regular cardinal.  
Let $\mci$ be an ideal on $\omega$ extending $[\omega]^{<\aleph_0}$ and $\bar{m}$ a fast sequence. 
We say that the Boolean algebra $\ba$ has the $(\kappa, \mci, \bar{m})$-c.c.  when: 
given $\langle \ma_\alpha : \alpha \in \uu_2 \rangle$ with $\uu_2 \in [\kappa]^\kappa$ 
a sequence\footnote{or renaming, without loss of generality, $\uu_2 = \kappa$.}  
 of elements of $\ba^+$, we can find 
$j < \omega$, $\uu_1 \in [\kappa]^\kappa$ and $A \in \mci$ such that: 

\begin{quotation}
\noindent $\oplus$  for every $n \in \omega \setminus A$ and every finite $u \subseteq \uu_1$ and every $i < n - j$, if 
\[ \frac{m_n}{(m^\circ_n)^{n^i}} < |u| \leq m_n \]
then there is some $v \subseteq u$ such that 
\[ |v| \geq \frac{|u|}{(m^\circ_n)^{n^{i+j}}} ~~\mbox{ ~~and~~ }~~ \bigcap \{ \ma_\alpha : \alpha \in v \}  > 0_\ba. \]
\end{quotation}
\end{defn}
\noindent For more on this chain condition, see \cite{MiSh:1167} Discussion 8.4. Its importance for omitting types is explained by the following, which informally says 
that if we construct our Boolean algebra by induction beginning from the completion of a free Boolean algebra in such a way as to maintain the $\kim$-c.c., then 
any resulting ultrafilter will not be good for theories $T_\xn$ whose level function is $1$ on a set which is not zero modulo the ideal $\mci$. 

\begin{fact}[\cite{MiSh:1167} Lemma 9.4] 
Suppose $\mu < \kappa \leq \lambda$ are infinite cardinals, with $\kappa$ regular. 
Suppose $\ba$ is a complete Boolean algebra, $\ba_* = \ba^1_{\kappa, \mu, \aleph_0}$, $\ba_* \lessdot \ba$ and 
$\ba$ has the $\kim$-pattern transfer property, see $\ref{paircc}$ below. Let $\xi$ be any level function such that $\xi^{-1}\{1\} \neq \emptyset \mod \mci$. Let 
$T = T_\xn$ where $\xn = \xn[\bar{m}, \bar{E}, \xi]$. Let $\de_{\oplus}$ be any nonprincipal ultrafilter on $\ba$. Then there is a possibility pattern for $T_n$ which 
has no multiplicative refinement.  

Hence if $\de$ is any regular ultrafilter on $\lambda$ built by separation of variables from $(\ba, \de_\oplus)$, 
$\de$ will not be good for $T_\xn$. 
\end{fact}

The pattern transfer property is often a simpler way to check that the $\kim$-c.c. has transferred to a later Boolean algebra in an inductive construction. 
We will use it in Lemma \ref{l:add-tnk-solution} below. 

\begin{defn}[Pattern transfer property] \label{paircc}  Let $\kappa$ be an uncountable cardinal, $\mci$ an ideal on $\omega$ extending 
$[\omega]^{<\aleph_0}$, and $\bar{m}$ a fast sequence. 
 The pair $(\ba_1, \ba_2)$ has the $(\kappa, \mci, \bar{m})$-pattern transfer property when: 
 $\mathbf{(1)}$ $\ba_1$ and $\ba_2$ are both complete Boolean algebras, 
$\mathbf{(2)}$ $\ba_1$ satisfies the $\kappa$-c.c
\footnote{we don't ask that $\ba_1$ have the $\kim$-c.c., only the 
$\kappa$-c.c., though in every application in the paper, $\ba_1$ will have the $\kim$-c.c.}, 
$\mathbf{(3)}$ $\ba_1 \lessdot \ba_2$, 
and 
$\mathbf{(4)}$ whenever $\uu_2 \in [\kappa]^\kappa$ and $\bar{\ma}^2 = \langle \ma^2_\alpha : \alpha \in \uu_2 \rangle$ 
is a sequence of 
elements of $\ba^+_2$, we can find a quadruple 
$(j, \uu_1, A, \bar{\ma}^1 )$ such that: 

\begin{enumerate}
\item[(a)] $j<\omega$
\item[(b)] $\uu_1 \in [\uu_2]^\kappa$
\item[(c)] $A \in \mci$
\item[(d)] $\bar{\ma}^1 = \langle \ma^1_\alpha : \alpha \in \uu_1 \rangle$ is a sequence of distinct elements of $\ba^+_1$ 
\item[(e)] $\alpha \in \uu_1$ implies $\ma^1_\alpha \bp \ma^2_\alpha$
\item[(f)] (i) implies (ii) where:
\begin{enumerate}
\item[(i)] we are given $n \in \omega \setminus A$, $i + j < n$, $u \subseteq \uu_1$, and $\ma_* \in \ba^+_1$ such that 
$m_n / (m^\circ_n)^{n^i} < |u| < m_n$ and 
\[ \ba_1 \models \ma_* \leq \bigcap_{\alpha \in u} \ma^1_\alpha \]
\item[(ii)] there is $v$ such that $v\subseteq u$ and $|v| \geq |u| / (m^\circ_n)^{n^{i+j}}$ {and} 
\[ \ba_2 \models \bigcap_{\alpha \in v} \ma^2_\alpha ~ \cap ~ \ma_* > 0.  \]
\end{enumerate}
\end{enumerate}
\end{defn}

\noindent The main work of the section is in the following lemma. 

\begin{lemma} \label{l:add-tnk-solution}
Assume $\mfa \in \AP$ and $\mfb$ is a $(\theta, T, \bar{\mb})$-extension of $\mfa$ where $\theta \leq \lambda$, $T= T_{n,k}$ for $n > k \geq 2$, and $\bar{\mb}$
is a possibility pattern arising from a type of the form 
\[  \{ R(x,\bar{a}_{v_\beta}) : \beta < \theta \}        \]
where each $|v_\beta| = k$. Then $\ba_\mfb$ satisfies the chain condition $\ref{d:1167-cc}$. 
\end{lemma}

\begin{proof}
The proof strategy will build on various ideas from \cite{MiSh:1167} \S 10.
However, to make this self-contained, we give all the details. 

Suppose that we are given $\bar{\mb} = \langle \mb_u : u \in [\theta]^{<\aleph_0} \rangle$, a sequence of elements, representing the problem from 
$\ba_\mfa$ just solved, and $\langle \mb^1_\alpha : \alpha < \theta \rangle$ the elements of $\de_\mfb$ which are its formal solution.  As written, consider this possibility 
pattern as coming from a positive $R$-type of the form $p(x) = \{ R(x,\bar{a}_{v_\alpha}) : \alpha < \theta \}$ in some enveloping ultrapower, 
where the notation means that in the enveloping ultrapower there is a sequence 
$\langle a_\gamma : \gamma < \theta \rangle$ of parameters, each $v_\alpha \in [\theta]^{k}$, and $\bar{a}_{v_\alpha} = \{ a_\gamma : \gamma \in v_\alpha \}$. 
[In slight abuse of notation, we will think of $v_\alpha$ as a set rather than a sequence, and so also consider $\bar{a}_{v_\alpha}$ as a set, 
since $R$ is symmetric and irreflexive.] 
We can think of $\mb_u$ as the image in $\ba$ of the set $\{ t \in \lambda : M \models \exists x \bigwedge_{\alpha \in u} R(x,\bar{a}_{v_\alpha}[t]) \}$. 
Note that since $p$ is a type, 
given any $\alpha_0, \dots, \alpha_{n-1}$ the set $\{ \{ \gamma : \gamma \in v_{\alpha_i} \} : i < n \}$ is never  
equal to $[w]^k$ for $w \in [\theta]^n$. Of course, this does 
not a priori prevent $\{ \{ a_\gamma[t] : \gamma \in v_{\alpha_i} \} : i < n \}$ from representing all $k$-element subsets of some $n$-element set in some index model 
with respect to the enveloping ultrapower; this is what we will have to avoid below by dealing not only with elements but with 
their ``collapse''. 
We will fix  p, the $a_\gamma$'s, and the $v_\beta$'s and refer to them throughout the proof. 

It will be very useful to keep track of equalities. For $\gamma < \theta$, call an element $\mx \in \ba^+_\mfa$ \emph{collapsed} for $\gamma$ if for some $\beta \leq \gamma$, 
\begin{equation} \label{e:collapse}
0 < \mx \leq \ma[a_\gamma = a_\beta] \mbox{ but for all } \delta < \beta, \mx \cap \ma[a_\gamma = a_\delta] = 0_{\ba_\mfa}. 
\end{equation}
Towards proving the pattern transfer property, suppose we are given $\langle \ma^2_\alpha : \alpha < \kappa \rangle$ with each $\ma^2_\alpha \in \ba^+_\mfb$. 
Following the normal form lemma \ref{smooth}, 
for each $\alpha < \kappa$ we may choose $\ii_\alpha = ( \mx_\alpha, u_\alpha, n_\alpha, \langle u_{\alpha, \ell} : \ell < n_\alpha \rangle)$ 
so that $\mx_\alpha \leq \mb_{u_\alpha}$ for each $\alpha < \kappa$, and 
\begin{equation} 
\label{1-dot} \ba_\mfb \models  0 < \mx_\alpha \cap \mb^1_{u_\alpha} \cap \bigcap_{\ell < m_\alpha} (- \mb^1_{u_{\alpha,\ell}})  \leq \ma^2_\alpha. 
\end{equation}
Without loss of generality, for each $\alpha < \theta$, $\mx_\alpha$ is collapsed for $\{ a_\gamma : \gamma \in v_\rho, \rho \in u_\alpha \}$. 
To record the effect of collapse, define a function $f: \theta \times \theta \rightarrow \theta$ by $(\alpha, \gamma) \mapsto \beta$ where $\beta$ is from 
equation (\ref{e:collapse}) in the case that $\mx = \mx_\alpha$. (So $\beta$ is well defined, and $\beta \leq \alpha$.)

Since $\kappa$ is regular, and by the $\Delta$-system lemma, we can find $\uu \in [\kappa]^\kappa$ such that: 
\begin{enumerate}
\item $\langle u_\alpha : \alpha \in \uu \rangle$ forms a $\Delta$-system.
\item $\langle P_\alpha : \alpha \in \uu \rangle$ forms a $\Delta$-system, where  $P_\alpha = \{ \gamma \in v_\rho, \rho \in u_\alpha \}$.\footnote{i.e. indices for all elements of $\langle a_\gamma : \gamma < \theta \rangle$ occurring as parameters in $R(x,\bar{a}_{v_\rho})$ for $\rho \in u_\alpha$.}
\item $\langle C_\alpha : \alpha \in \uu \rangle$ forms a $\Delta$-system with heart $C_*$, \\ where 
$C_\alpha = \{ f(\alpha, \gamma) :  \gamma \in v_\rho, \rho \in u_\alpha \}$. \footnote{Recall that 
$\mx_\alpha$ is collapsed for all of the parameters in $P_\alpha$. 
$C_\alpha$ is the set of indices for elements to  
which the elements of $P_\alpha$ collapse on $\mx_\alpha$.  A priori, we may have $|C_\alpha| \leq |P_\alpha|$ if parameters from 
different formulas collapse to the same value.  However, $\{ f(\alpha, \gamma) : \gamma \in v_\rho \}$ is always a set of size $k$, since 
$\mx_\alpha \leq \mb_{u_\alpha} \leq \mb_{ \{ \rho \} }$ for each $\rho \in u_\alpha$, and $R$ is irreflexive.}
\item $\langle g_\alpha \rstr [C_*]^k  : \alpha \in \uu \rangle$ is constant, where $g_\alpha: [C_\alpha]^k \rightarrow \{ 0, 1 \}$ is given as follows. 
If there is some $\rho \in u_\alpha$ such that $\{ f(\alpha, \gamma) : \gamma \in v_\rho \} = \{ \beta_0, \dots, \beta_{k-1} \}$, then 
$g_\alpha( \{ \beta_0, \dots, \beta_{k-1} \}) = 1$, else  $g_\alpha( \{ \beta_0, \dots, \beta_{k-1} \}) = 0$. \footnote{Because of the collapse, the 
key parameters in play for $\mx_\alpha$ are $\{ a_\gamma : \gamma \in C_\alpha \}$. As an example, if $k=2$, it could be that 
$u_\alpha = \{ \rho, \zeta \}$ and $v_\rho = \{ \gamma_1, \gamma_2 \}$ and $v_\zeta = \{ \gamma_3, \gamma_4 \}$. So 
$R(x,\bar{a}_{v_\rho}) = R(x, a_{\gamma_1}, a_{\gamma_2})$ and $R(x,\bar{a}_{v_\zeta}) = R(x, a_{\gamma_3}, a_{\gamma_4})$.  
Suppose $f(\alpha, \gamma_1) = \delta_1 = f(\alpha, \gamma_4)$, $f(\alpha, \gamma_2) = \delta_2$ and $f(\alpha, \gamma_3) = \delta_3$. 
Then even though $\{ \delta_2, \delta_3 \}$ is a two-element subset of $C_\alpha$, neither of our formulas amounts (on $\mx_\alpha$) 
to asserting a connection to $a_{\delta_2}$ and $a_{\delta_3}$. 
So we would have $g_\alpha(\{ \delta_2, \delta_3 \} ) = 0$ whereas, 
for instance, $g_\alpha( \{ \delta_1, \delta_2 \} ) = 1$. Since the type is positive, $0$ here asserts ``no information,'' not negation, which will suffice.}
\end{enumerate}
First let us prove that for any $v \in [\theta]^{<\aleph_0}$, 
\begin{equation}  \label{e:32}
  \bigcap_{\alpha \in v} \mx_\alpha \leq \mb_w 
\end{equation}
where $w = w(v) = \bigcup_{\alpha \in v} u_\alpha$. 
Suppose not. Let $C = \bigcup_{\alpha \in v} C_\alpha$. Then there must be a set $A$ and a function $G$ such that:
\begin{itemize}
\item $A \subseteq C$, $|A| = n$,
\item $G: [A]^k \rightarrow v$ is such that if $B \in [A]^k$ and $G(B) = \alpha$ then  
$B \subseteq C_\alpha$ and $g_\alpha(B) = 1$,  
\item and 
\[ \my := \bigcap_{\alpha \in v} \mx_\alpha \cap \bigcap_{B^\prime \in [A]^{k+1}} \ma[R(B^\prime)]~~ > 0. \]
\end{itemize}
Informally, there is an $n$-tuple of parameters (after collapse) such that $R$-connecting to each $k$-element subset is implied by some $\alpha$, and
there is additionally positive intersection with the ($\ba$-image of a) set where each $(k+1)$-element subset has an $R$-edge, \emph{thus} altogether 
causing $x$ to form a forbidden configuration over $B$.  

Recall that $C_*$ was the heart of the $\Delta$-system $\langle C_\alpha : \alpha \in \uu \rangle$. 

First observe that the range of $G$ cannot have size $1$, as each $\mx_\alpha \leq \mb_{u_\alpha}$.  Second, observe that we cannot have 
$A \subseteq C_*$. This is because $C_* \subseteq C_\alpha$ for each $\alpha$ and 
$\langle g_\alpha : \alpha \in \uu \rangle$ is constant on $[C_*]^k$, so for each $B \in [C_*]^k$, 
if $G(B) = \alpha$ for some $\alpha \in v$ then $g_\alpha(B)$ for every $\alpha \in v$ and so we have a contradiction already on a positive set 
for some (every) $u_\alpha$, contradicting $\mx_\alpha \leq \mb_{u_\alpha}$. 

Suppose for a contradiction that $A \not\subseteq C_*$. There are two cases. In the first case, suppose we can find $i \neq j$ such that $G(B_i) = \alpha_i$, 
$G(B_j) = \alpha_j$, and there are elements $b_i \in B_{i} \setminus C_{\alpha_j}$, $b_j \in B_j \setminus C_{\alpha_i}$, so necessarily 
$b_i, b_j \notin C_*$.  Then $G$ cannot be well defined on any 
$B_\ell \supseteq \{ b_i, b_j \}$, since if $G(B_\ell) = \alpha_\ell$, say, then either $\ell \neq i$ or $\ell \neq j$, say the first, and then 
$b_i \in C_{\alpha_i} \cap C_{\alpha_\ell}$ implies $b_i \in C_*$ by definition of $\Delta$-system, thus $b_i \in C_{\alpha_j}$, contradicting our choice of $b_i$. 
So this cannot happen.  
In the second case, for at most one $\alpha$ (though possibly more $B$s) does it happen that $G(B) = \alpha$ and $B \subseteq C_\alpha$, $B \not\subseteq C_*$. Then 
$A \subseteq C_\alpha$, contradicting $\mx_\alpha \leq \mb_{u_\alpha}$. 

We have ruled out all possible cases, so this proves equation (\ref{e:32}).

Let us verify \ref{paircc}(4)(f) holds for $A = \emptyset$, $j = 0$ when $\bar{\ma}^2$ is the sequence from the beginning of the proof and the role of 
$\bar{\ma}^1$ is played by $\langle \mx_\alpha : \alpha < \theta \rangle$. 
Suppose we are given $n \in \omega \setminus A$, a finite $u \subseteq \uu$, and a nonzero $\ma_* \in \ba^+_\mfa$ such that 
$m_n / (m^\circ_n)^{n^i} < |u| < m_n$ and \begin{equation} \label{e155}
 \ba_\mfa \models 0 < \ma_* \leq \bigcap_{\alpha \in u} \ma^1_\alpha.  
\end{equation}
To fulfill the $\kim$-pattern transfer (in a quite strong way) it would suffice to show that already for 
$v = u$,
\begin{equation} \label{e156a} \ba_2 \models \bigcap_{\alpha \in v} \ma^2_\alpha ~ \cap ~ \ma_* > 0.  
\end{equation}
By the comment on normal form \ref{smooth}(4), to prove (\ref{e156a})  it would suffice to show that  
\begin{equation} \label{e156}
 \ba_\mfb \models \bigcap_{\alpha \in v} (\mx_\alpha \cap \mb^1_{u_\alpha}) \cap \ma_* > 0. 
\end{equation} 
For this, in turn, it suffices to show that $\bigcap_{\alpha \in v} \mx_\alpha \leq \mb_w$, where $w = w(v)$, 
as we have already done in equation (\ref{e:32}).   Why is this sufficient? 
We can define $\hat{h}_w $ an endomorphism from 
$\ba_\mfb$ onto $\ba_\mfa$ which is the identity on $\ba_\mfa$ and for each $\alpha \in v$ takes  
$\mx_\alpha \cap \mb^1_{u_\alpha} = \mx_\alpha \cap \bigcap_{\gamma \in u_\alpha} \mb^1_\gamma$ 
to $\mx_\alpha \cap \mb_w$.  Now $\ma_* \leq \bigcap \mx_\alpha \leq \mb_w$ tells us that 
\[  \hat{h}_w\left(\ma_{*} \cap \bigcap_{\alpha \in v} (\mx_\alpha \cap \mb^1_{u_\alpha})\right) = 
\ma_{*} \cap \bigcap_{\alpha \in v} \mx_\alpha \cap \mb_w = \ma_{*} > 0 \]
and equation (\ref{e156}) follows. 

This completes the proof of the $\kim$-pattern transfer, and so the proof of the Claim. 
\end{proof}

\begin{theorem} \label{t:dir2}  Suppose $\mu < \kappa \leq \lambda$ are infinite cardinals and $\kappa$ is regular. Suppose $T_0 = T_{n,k}$ for 
some $2 \leq k < n < \omega$ and suppose $T_1 = T_\xn$ where $\xn = \xn[\bar{m}, \bar{E}, \xi]$.  Then there is a regular ultrafilter $\de$ on $\lambda$ such that 
$\de$ is good for $T_0$ and not for $T_\xn$. 
\end{theorem}

Now we explain how this result easily gives something more general. 
In \cite{MiSh:1167} a family of continuum many theories $\{ T_\xm : \xm \in \bM_* \}$ were produced and it was shown that for 
an partition of $\bM_*$ into $\mcm$ and $\mcn$, or just $\mcm$, $\mcn$ any two disjoint subsets of $\bM_*$, it is possible to construct 
$(\ba, \de)$ satisfying this c.c. which is good for the theory of the random graph and for all $T_\xm$ $(\xm \in \mcm)$ 
but for no $T_\xn$ ($\xn \in \mcn$).  These theories had in common $\bar{m}$ and $\bar{E}$ and differed on $\xi$ which was chosen from an 
independent family of functions in the sense of Engelking-Kar\l owicz \cite{ek} (see \cite{MiSh:1167}, Fact 6.20 and Corollary 6.22).  
The choice of $\mcm, \mcn$ would affect the choice of 
ideal $\mci$ used in \ref{d:1167-cc}, and this was in fact a characterization, see \cite{MiSh:1167} Theorem 11.10.  

In our context,  Lemma \ref{l:add-tnk-solution} means that we can 
ensure any such $\de$ is, in addition, good for $T_{n,k}$, simply by interleaving into the inductive construction steps which 
handle positive $T_{n,k}$-types. So in fact we get $T_\xn \not\tlf T_{n,k}$ for any $\xn \in \bM_*$, 
along with naturally stronger statements reflecting that we can deal with many $(n,k)$ and also freely partition $\bM_*$, as we state now.

\begin{concl} \label{c:6.6} 
Suppose $\mu < \kappa \leq \lambda$ are infinite cardinals and $\kappa$ is regular. 
Let $\{ T_\xm : \xm \in \bM \}$ be the sequence of continuum many independent theories from \cite{MiSh:1167}. 
For any disjoint $\mcm, \mcn \subseteq \bM$ there is a regular ultrafilter $\de$ on $\lambda$ such that 
$\de$ is:
\begin{enumerate}
\item good for the theory of the random graph, 
\item good for every $T_\xm$ with $\xm \in \mcm$, 
\item not good for any $T_\xn$ with $\xn \in \mcn$, 
\item good for $T_{n,k}$ for every $2 \leq k < n < \omega$. 
\end{enumerate}
\end{concl}

\begin{proof}
Follow the proof of \cite{MiSh:1167} Conclusion 10.25, except that in step 4 of that proof 
also interleave adding solutions to positive $R$-types for any or all theories $T_{n,k}$ as 
desired. This preserves pattern transfer by Lemma \ref{l:add-tnk-solution}, and thus the $\kim$-c.c., which suffices.  
\end{proof}

\begin{rmk}
We may replace $\ref{c:6.6}(4)$ by, for a given $\bk \geq 2$, 
\begin{enumerate} 
\item[(4)$^\prime$] if $2 \leq n <\omega$ then $\de$ is good for $T_{n,k}$ iff $k \geq \bk$.
\end{enumerate}
\end{rmk}

\vspace{3mm}

\section{Flexibility and the c.c.}  \label{s:flex-cc}

In this section we show that no regular ultrafilter built by separation of variables where the Boolean algebra $\ba$
has the $\mu$-c.c. for some regular uncountable $\mu < \lambda$ can be flexible. 
This conclusively settles a question addressed for free Boolean algebras in various earlier theorems. \footnote{\cite{Sh:c} Claim 3.23 
p. 364 proves that ultrafilter built by independent families of functions of small range will not be good 
(really, will not be flexible, avant la lettre).  Another proof for independent families was given in 
\cite{MiSh:997} \S 5, and updated and applied to separation of variables in the non-saturation argument of 
\cite{MiSh:999} Corollary 8.9, still only for free Boolean algebras.}  
By contrast, here the only assumptions on $\ba$ are the c.c. and  
those inherited from separation of variables: $\ba$ is complete, of size $\leq 2^\lambda$, and [instead of the 
$\lambda^+$-c.c. in the usual separation of variables] has the $\mu$-c.c. for some regular uncountable $\mu < \lambda$. 

This is a nice occasion to re-recall with appreciation that Kunen some years ago, on reading the definition of 
flexibility in \cite{mm4}, brought to our (MM's) attention the definition ``OK'' (as noted in e.g. \cite{MiSh:996}), thus setting the stage for some 
nice connections to Dow's problem of constructing ultrafilters which are OK (flexible) and not good. The first open problem in the 
next section relates to this. 

Let us recall the definition of flexibility. Suppose $\{ X_\alpha : \alpha < \lambda \}$ is a regularizing family in the ultrafilter $\de$ on $I$, $|I| = \lambda$, that is, 
the intersection of any finitely many $X_\alpha$'s belongs to $\de$ but the intersection of any infinitely many $X_\alpha$'s is empty.  
This is equivalent to saying that any $t \in I$ belongs to only finitely many $X_\alpha$'s. So we can naturally associate a ``size'' to any such regularizing family, 
namely, the nonstandard natural number $\prod_t n_t/\de$ where $n_t = | \{ \alpha < \lambda : t \in X_\alpha \} |$. Flexibility asks, informally, for regularizing families of 
arbitrarily small (nonstandard) size.

\begin{defn}[\cite{mm4}, Definition 8.2] 
Say the regular ultrafilter $\de$ on $I$, $|I| = \lambda$ is \emph{flexible} if for any $\de$-nonstandard $n_*$, there is a regularizing family $\{ X_\alpha : \alpha < \lambda \}$ 
with $| \{ \alpha < \lambda : t \in X_\alpha \} | \leq n_*[t]$ for $\de$-almost all $t \in I$.
\end{defn}

We now prove the section's main theorem. 

\begin{rmk}
Theorem $\ref{small-cc}$ is optimal: already when $\ba$ is free, if the ultrafilter on $\ba$ is $\aleph_1$-complete, the resulting 
regular ultrafilter can be flexible, \cite{MiSh:996} Corollary 6.3. 
\end{rmk}

\begin{theorem}  \label{small-cc}
Suppose $\mu$ is regular and uncountable. 
Suppose $\de_1$ is built from $(\de_0, \jj, \ba, \de)$ where $\de_0$ is an excellent regular filter on $\lambda$, 
$\ba$ satisfies the $\mu$-c.c., and $\de$ is not $\aleph_1$-complete.  Then $\de_1$ is not $\mu$-flexible. 
\end{theorem}

\begin{proof}
As $\ba$ is not $\aleph_1$-complete, there is a maximal antichain $\langle \ma_n : n < \omega \rangle$ with each $\ma_n \notin \de$. 
We may define a flexibility problem $f$ by: for each $u \in [\mu]^{<\aleph_0}$, let $f(u) = \bigcup \{ \ma_\ell : \ell > |u| \}$.

Assume for a contradiction that $\langle \mb_\alpha : \alpha < \mu \rangle$ is a solution of $f$. 
For each $\alpha <\mu$, for some $n = n(\alpha) < \omega$, 
\[ \mb^\prime_\alpha = \mb_\alpha \cap \ma_n > 0. \]
Define $S_n = \{ \alpha : n(\alpha) = n \}$. So 
\begin{equation}
\label{eq-n}
\mbox{if } u \subseteq S_n \mbox{ and }\bigcap_{\alpha \in u} \mb^\prime_\alpha > 0 \mbox{ necessarily }|u| \leq n.
\end{equation}  As $\mu$ is regular and uncountable, for some $m$ we have that $|S_{m}| = \mu$.  
Fix this $m$ for the rest of the proof. 

Now we try to choose $\alpha_\ell \in S_{m}$ by induction on $\ell$ so that $\langle \alpha_i : i < \ell \rangle$ is strictly increasing 
(or just without repetition) so that $\bigcap_{i<\ell} \mb^\prime_{\alpha_i} > 0$ and so that 
\[ S_{m, \ell} = \{ \beta ~ :~   \bigcup_{i < \ell} \alpha_i < \beta, \beta \in S_m \mbox{ and } 
\mb^\prime_\beta \cap \bigcap_{i<\ell} \mb^\prime_{\alpha_i}   > 0 \} \mbox{ has size $\mu$.} \]
For $\ell=0$ this holds trivially: $S_{m,0} = S_m$ which has size $\mu$.   Assume $\alpha_0, \dots, \alpha_{\ell-1}$ have been chosen. 
For each $\beta \in S_{m,\ell-1}$ let $\mb^{[\ell]}_\beta = \mb^\prime_\beta \cap \bigcap_{i<\ell} \mb^\prime_{\alpha_i}$.
Suppose that for some $\beta \in S_{m,\ell-1}$ 
there exist $\mu$ elements $\gamma \in S_{m,\ell-1}$ with 
\[ \mb^{[\ell]}_\beta \cap \mb^{[\ell]}_\gamma > 0. \]
If there is such a $\beta$, it can serve as $\alpha_\ell$ and we finish the inductive step. If not, we can build an antichain of 
size $\mu$ as follows: for every $\beta \in S_{m,\ell-1}$ there is 
$\gamma_\beta \in (\beta, \mu)$ such that 
\[ \gamma_\beta \leq \gamma \in S_{m,\ell-1} \implies \mb^{[\ell]}_\beta \cap \mb^{[\ell]}_\gamma = 0_{\ba}. \]
Let $E_1 = \{ \delta < \mu : \delta $ limit and $ (\forall \beta \in S_{m,\ell-1} \cap \delta) (\gamma_\beta < \delta) \}$, so 
$E_1$ is a club.  Also $E_2 = \{ \delta < \mu : \delta $limit$, \delta = \operatorname{sup}(S_{m,\ell-1} \cap \delta) \}$ 
is a club of $\mu$ because $|S_{m,\ell-1}|=\mu$ by inductive hypothesis. Let $E = E_1 \cap E_2$. For $\delta \in E$, let 
$\beta_\delta = \min (S_{n,\ell} \setminus \delta)$, so $\beta_\delta < \min (E \setminus (\delta+1))$. So 
$\langle \mb^{[\ell]}_{\beta_\delta} : \delta \in E \rangle$ is an antichain of $\ba$ of cardinality $\mu$, contradicting the 
$\mu$-c.c. 

This contradiction shows we can carry the induction for any finite $\ell$, but for $\ell>n$ we contradict equation 
(\ref{eq-n}), which completes the proof. 
\end{proof}

\begin{disc} \label{d35}
\emph{Which are the functions $f \in {^\lambda \omega}$ for which there is a regularizing family? They are the functions such that 
for all $n<\omega$, $f^{-1}\{ n \} = \emptyset \mod \de_1$. Recall that the set of such $f$s is a co-initial segment of 
$\mathbb{N}^\lambda/\de$.}
\end{disc}

\begin{disc} \label{d:infty-class}
\emph{Thus for any $\kappa^+ < \lambda$, we have a natural notion of the ``$\infty$-class'', that is, the theories $T$ for which there is no 
$(\kappa^+, T)$-moral $\de \in \operatorname{uf}(\ba)$ when $\ba$ has the $\kappa^+$-c.c. 
This includes all non-low simple theories and all non-simple theories.}  
\end{disc}

\begin{disc} \emph{Thesis: we should now look at morality before considering goodness.  Also, the case of the free Boolean algebras 
used in various earlier papers, while fine and interesting, now seem perhaps less natural in the big picture.}
\end{disc}

\begin{rmk}
Prior to Theorem $\ref{small-cc}$, our state of knowledge on the ``saturation separation'' of theories like $T_{n,k}$ from non-simple theories like 
$\tfeq$ was less clear, i.e.,  looking for regular ultrafilters where the first kind of theories are $\kappa$-saturated and the second set 
not $\mu$-saturated for $\kappa > \mu$ are there any a priori constraints on the distance of $\lambda$ and $\mu$. 
Now we see that indeed they can be arbitrarily far apart, although we still do not know that the $T_{n,k}$'s are always above $\tfeq$, see 
$\ref{q:tfeq}$ below. 
\end{rmk}

\begin{rmk}
Theorem $\ref{small-cc}$ also seems relevant to the questions raised in \cite{MiSh:1167} Discussion $13.7$ on  
the possible variation of regularizing families in regular ultrafilters. 
\end{rmk}

\br

\section{Some open questions}

\begin{qst} \label{q:tfeq}
Is $T_{n,k} \tlf T_{\operatorname{feq}}$?
\end{qst}

Either direction could be interesting. A positive answer in ZFC could make progress towards 
proving that simplicity is a ZFC dividing line in Keisler's order (we know by \cite{MiSh:1030} that it is a dividing line assuming 
existence of a supercompact cardinal). A negative answer in ZFC might address a 1985 problem of Dow \cite{dow} by providing an 
example of an ultrafilter which is more flexible than good. This is because any regular ultrafilter which is good for $T_{\operatorname{feq}}$  
must be flexible \cite{mm4}, but if it is not good for some $T_{n,k}$, it is not good. 
See also \cite{MiSh:1070}, \S 1, for a precise discussion of Dow's problem, stated as Question 1.3 there.  

\begin{qst}
Among the ``$<\infty$'' theories in the sense of Discussion $\ref{d:infty-class}$, is there a maximal one? 
\end{qst}

\begin{qst}
Are the theories $T_{n,k}$ incomparable with each other as $n,k$ vary? 
\end{qst}

\begin{qst}
Is Keisler's order absolute?
\end{qst}

\vspace{3mm}

\newpage

\section{Appendix: Existing evidence for independence} \label{s:indep}

The main constructions of this paper have shown that the theories $T_\xm$ and $T_{n,k}$ are incomparable in Keisler's order. To exposit 
some of the considerations in saturating these theories, we sketch two proofs we could have given in this vein by assembling existing results in 
 a perhaps less satisfactory way (and one which does not address Question \ref{q0.2}).   They may be of interest to readers who have followed 
 prior work, or who are thinking about how to extend existing techniques. However, to be clear, the arguments below are 
superceded by the main theorems just given. 

\subsection*{Analysis of the $T_\xm$s.} 
On one hand, any regular ultrafilter which is flexible and good for the theory of the random graph is good for any $T_\xm$, by \cite{MiSh:1206} Corollary 5.4. 
This proof proceeds by isolating a finite cover property-like property of the theories $T_\xm$, \cite{MiSh:1206} Definition 4.1, and 
showing that the combination of flexibility to handle this fcp-like property, along with goodness for $\trg$, the minimum unstable theory, to handle the independence property,
is sufficient to deal with $T_\xm$.\footnote{This appears to be a strong statement about the simplicity of the $T_\xm$'s, but may in fact be a statement about the 
ability of forking, here appearing via flexibility, to ``drown out'' the complexity of independence.} 

On the other hand, if $\ba = \ba^1_{\alpha_*, \mu, \aleph_0}$ and $\de$ is a nonprincipal ultrafilter on $\ba$ 
then no subsequent ultrafilter built from $(\ba, \de)$ by separation of variables can be good for any 
$T_\xm$.  The prototype for this argument is \cite{MiSh:1140}, Lemma 3.2.  The deeper explanation is in  
\cite{MiSh:1167} Lemma 9.4. Essentially, what can threaten saturation of a $T_\xm$ is the following. Suppose we have, as part of a type in the ultrapower, 
a finite set of conditions each asserting that $x$ should $R$-connect to some $a_i$. Because this is a type in the ultrapower, the $a_i$'s may belong in the ultrapower 
to the same predicates, or compatible ones.  However, it could happen in the projection to the $t$-th index model that the $a_i[t]$'s fall across predicates 
$P_{\nu^\smallfrown \langle j \rangle}$ for a set of $j$'s which is ``large'' in the sense of $n = \lgn(\nu)$, and so it is not possible to connect there to all of them, 
or indeed even to many of them, simultaneously.  One reason why (completions of) a free Boolean algebra will therefore cause problems for this theory is explained by the 
chain condition \cite{MiSh:1167} 8.2, which holds of this $\ba$ by \cite{MiSh:1167} 8.5, as well as for a wider class of Boolean algebras which need not be free. 
Informally, it expresses a requirement that among any large set of elements (for varying finite notions of large) there is 
an only slightly less large subset with nonzero common intersection.  This c.c. was stated and applied in \S \ref{s:second}. 
(An alternative avenue is in the second argument on the next page.) 

\subsection*{Analysis of the $T_{k+1,k}$s.}  Recall that Kuratowski and Sierpi\'{n}ski characterized the $\aleph_n$'s by existence of free sets in 
set mappings, see \cite{ehmr} or \cite{MiSh:1050} \S 1, or further details in \S \ref{s:n-s}.  
This turns out to have a close connection to the problem of saturating the theories $T_{k+1,k}$.  These theories can be thought of as 
posing fairly pure amalgamation problems: if we let $\bar{a} = \langle a_0, \dots, a_k \rangle$ be a sequence of distinct elements, 
and let $\bar{a}_i$ denote the $k$-element subsequence 
in which $a_i$ is omitted, then $\{ R(x,\bar{a}_i) : 0 \leq i \leq k \}$ is consistent if and only if 
$\neg R(\bar{a})$.  But even if this forbids certain instances of $R$ on parameters for the type in the ultrapower, 
it may be that an instance of $R$ appears on $\{ a_0[t], \dots, a_k[t] \}$ in the projection to the $t$-th index model, thus 
blocking realization there.  Whether the appearance of such an $R$ can be controlled by conditions on the strictly smaller $\bar{a}_i$ is similar to the question of 
existence of a set of size $(k+1)$ in a set mapping which is free, informally, not controlled by any of its $k$-element subsets.  

\cite{MiSh:1050} carried out this idea fully to prove the following. 
Suppose $\de$ is a regular ultrafilter on $\lambda \geq \mu \geq \aleph_0$ built by separation of variables and $\ba = \ba^1_{2^\lambda, \mu, \aleph_0}$. 
Suppose that in addition $\lambda = \aleph_{\alpha + \ell}$ and $\mu = \aleph_\alpha$ and $2 \leq \ell < \omega$.  Then, for $k < \ell$, $\de$ is not good for $T_{k+1,k}$.  
Here $n = k+1$ is used in an essential way: the bottleneck is \cite{MiSh:1050} Subclaim 5.3. 
However, $\de$ may be chosen to be good for all $T_{k+1,k}$ where $k > \ell$.  The so-called $(\lambda, \mu)$-perfected ultrafilters of \cite{MiSh:1050}
Definition 3.10-3.11 suffice. 
Informally, these are ``as good as possible'' modulo the constraint that $\ba$ has a much smaller c.c., $\mu^+$, than the $\lambda^+$ which 
would be expected in $\mcp(\lambda)$.  

Let us now assemble this understanding to give two proofs of incomparability: 

\subsection*{First argument: $T_\xm$ and $T_{k+1,k}$ are incomparable, using a large cardinal.}   This is \cite{MiSh:1206}, Conclusion 5.8, but we sketch the argument, 
which builds on the first (non-ZFC) incomparable classes in Keisler's order (see the independent works \cite{ulrich} and \cite{MiSh:1124}).  
Let $(\lambda, \mu, \theta, \sigma)$ be suitable, so in particular they are decreasing in size, and suppose $\sigma$ is uncountable and supercompact. 
Suppose $\lambda = \aleph_{\alpha + \ell}$ and $\mu = \aleph_{\alpha}$.  Suppose $\ba = \ba^1_{2^\lambda, \mu, \theta}$. Suppose $\de$ is a so-called  
\emph{optimized} ultrafilter built with this data by separation of variables. Then $\de$ is flexible and good for the random graph, so is good for any $T_\xm$. 
However, $\de$ is not good for $T_{k+1,k}$ for any $2 \leq k < \ell$.   This shows $T_{k+1,k} \not \tlf T_\xm$ for any such $k, \xm$. 
If instead of the large cardinal assumption we suppose $\theta = \sigma = \aleph_0$ and choose $\de$ to be $(\lambda, \mu)$-perfected, then 
$\de$ will not be good for any $T_{\xm}$, but it will be good for $T_{k+1,k}$ for any $k > \ell$. This shows $T_\xm \not \tlf T_{k+1,k}$. 

\subsection*{Second argument: $T_\xm$ and $T_{k+1,k}$ are incomparable in ZFC, using possibly uncountable $\theta$.}
We may give a proof of incomparability in ZFC by a more subtle attention to the cardinal $\theta$, the depth of intersection in $\ba = \ba^1_{2^\lambda, \mu, \theta}$. 
Suppose $\lambda = \aleph_{\alpha+\ell} > \mu = \aleph_\alpha \geq 2^{\aleph_0}$, and we will consider the cases $\theta = \aleph_0$ and $\theta = \aleph_1$ 
[here always $\sigma = \aleph_0$]. 
Suppose $\theta = \aleph_0$ and $\de$ is $(\lambda, \mu)$-perfected. Then for any $k > \ell$, $\de$ will be good for $T_{k+1,k}$, but it will not be good for 
$T_\xm$.  So $T_\xm \not \tlf T_{k+1,k}$. 

On the other hand, we may adapt the definition of $(\lambda, \mu)$-perfected ultrafilters to allow for uncountable $\theta$: 
see \cite{MiSh:1140}, Definition 5.3. 
Let $\de$ be perfected for our given $\lambda$, $\mu$, $\theta = \aleph_1$, and $\sigma = \aleph_0$. 
Then $\ell$, the integer distance between $\lambda$ and $\mu$, still controls the saturation or non-saturation of the $T_{k+1,k}$.  So $\de$ is not good for 
$T_{k+1,k}$ when $k < \ell$. 
To finish the analysis, it would be possible to show, by closely following the proof of  \cite{MiSh:1140} Theorem 5.5 [which deals with precursors to the $T_\xm$'s, called 
$T_f$] that $\de$ is good for $T_\xm$.   
So $T_{k+1,k} \not\tlf T_\xm$. 
 
\vspace{2mm}

The proofs we have chosen to focus on in the main text above strengthen this in several ways: they are in ZFC, 
they use $\theta = \aleph_0$, and they apply to $T_{n,k}$ for any $n > k \geq 2$, not requiring $n = k+1$.   

This said, what is to us 
even more interesting about the present paper is the substantially different approach, involving tailor-made Boolean algebras 
which are not free, and whose structure comes to reflect, via the chain conditions, certain 
aspects of the Boolean algebra of formulas in the theories in question. 
Although Keisler's order has long been about two kinds of Boolean algebras (the obvious one, and the one of formulas) and 
two kinds of ultrafilters on them (the obvious ones, and types), the interaction between the two remained mysterious. 
These constructions begin to close that gap.

\end{document}